\documentclass[letterpaper,11pt,reqno]{amsart} 
\usepackage[portrait,margin=3cm]{geometry} 
\usepackage{mathrsfs,xfrac} 
\usepackage[colorlinks=true,linkcolor=blue,citecolor=blue,urlcolor=blue]{hyperref} 
\usepackage{amsmath,amssymb,amsthm,amsfonts,amsbsy,latexsym,dsfont,color,graphicx,enumitem}
\usepackage[numeric,initials,nobysame]{amsrefs} 
\usepackage[foot]{amsaddr}

\newtheorem{thm}{Theorem}[section]
\newtheorem{lem}[thm]{Lemma}

\newtheorem{rem}[thm]{Remark}

\renewcommand{\le}{\leqslant} 
\renewcommand{\ge}{\geqslant} 

\newcommand{\ra}{\rangle}
\newcommand{\la}{\langle} 
\newcommand{\wt}{\widetilde}

\newcommand{\ind}{\mathds{1}}
\newcommand{\eps}{\varepsilon}

\newcommand{\norm}[1]{\left\Vert#1\right\Vert}
\newcommand{\abs}[1]{\left\vert#1\right\vert}

\newcommand{\ie}{\emph{i.e.,}}

\let\ga=\alpha \let\gb=\beta \let\gc=\gamma  
\let\gf=\varphi     \let\gl=\lambda         \let\gs=\sigma \let\gt=\tau 
  \let\gz=\zeta
 \let\gD=\Delta   
\let\gO=\Omega           
                                
\newcommand{\cA}{\mathcal{A}}

\newcommand{\cI}{\mathcal{I}}

\newcommand{\cR}{\mathcal{R}}

\newcommand{\cV}{\mathcal{V}}
  
\newcommand{\vone}{\mathbf{1}}

\newcommand{\mvR}{\boldsymbol{R}}

\newcommand{\dC}{\mathds{C}}

\newcommand{\dR}{\mathds{R}}

\newcommand{\dZ}{\mathds{Z}} 

\DeclareMathOperator{\E}{\mathds{E}}
\DeclareMathOperator{\pr}{\mathds{P}}

\DeclareMathOperator{\var}{Var}
\DeclareMathOperator{\cov}{Cov}

\newcommand{\wh}[1]{\widehat{#1}}
\newcommand{\ol}[1]{\overline{#1}}

\begin{document}
\title[Size of vacant set]{Fluctuation results for size of the vacant set for random walks on discrete torus}
\author[Dey]{Partha Dey}
\author[Kim]{Daesung Kim}

\address{University of Illinois at Urbana Champaign, 1409 W Green Street, Urbana, Illinois 61801}
\email{psdey@illinois.edu, daesungk@illinois.edu}

\date{\today}

\begin{abstract}
	We consider one or more independent random walks on the $d\ge 3$ dimensional discrete torus. The walks start from vertices chosen independently and uniformly at random. We analyze the fluctuation behavior of the size of some random sets arising from the trajectories of the random walks at a time proportional to the size of the torus. Examples include vacant sets and the intersection of ranges. The proof relies on a refined analysis of tail estimates for hitting time and can be applied for other vertex-transitive graphs.
\end{abstract}

\maketitle

\section{Introduction}
Consider finitely many independent random walks on the $d\ge 3$ dimensional discrete torus with large side length, each starting from vertices chosen independently and uniformly at random. We are interested in the mean and fluctuation behavior of the size of some random sets arising from the trajectories of the random walks at a time proportional to the size of the torus. Examples include the vacant set or the set of vertices not visited by any of the walks; the size of the intersection of ranges, among others.

In particular, we fix a positive integer $\ell\ge 1$, where $\ell$ denotes the number of independent random walks. The discrete torus is denoted by $\dZ_n^d=(\dZ/n\dZ)^{d}$, with side length $n$. Note that $n$ controls the size of the graph and we are interested in the large $n$ limit with a fixed $d$. We consider $\ell$ many independent $\frac12$\ndash lazy random walks $(X_{t,i})_{t\ge 0}, i\in[\ell]:=\{1,2,\ldots,\ell\}$ starting from $X_{0, i}=\xi_{i}, i\in[\ell]$, respectively, where $\xi_{i}$'s are i.i.d.~uniformly distributed over $\dZ_n^d$. We will use $\pi=\pi_{n}$ to denote the uniform distribution over $\dZ_n^d$.

We define the range of the $i$-th random walk at time $t$ as
\begin{align}
	\cR_{i}(t):=\{X_{s,i}\mid s=0,1,\ldots,t\} \quad \text{ for } i=1,2,\ldots,\ell;
\end{align}
the size of vacant set, \ie\ the number of vertices not visited by any of the $\ell$ walks at time $t$ as
\begin{align}\label{def:v}
	V_{n}^{(\ell)}(t):=\abs{\dZ_{n}^{d}\setminus\cup_{i=1}^{\ell}\cR_{i}(t)},
\end{align}
and the size of the intersection of ranges at time $t$ as
\begin{align}\label{def:r}
	R_{n}^{(\ell)}(t):=\abs{\cap_{i=1}^{\ell}\cR_{i}(t)}.
\end{align}
If we define $\tau_{i}(v)$ as the hitting time at the vertex $v$ for the $i$-th random walk, \ie
\begin{align*}
	\tau_{i}(v):=\min\{t\ge 0\mid X_{t,i}=v\},
\end{align*}
then we have for all $t\ge 0$
\begin{align*}
	V_n^{(\ell)}(t) = \sum_{v\in\dZ_n^d} \prod_{i=1}^{\ell}\vone_{\{\tau_{i}(v)>t\}}
	\text{ and }
	R_n^{(\ell)}(t) = \sum_{v\in\dZ_n^d} \prod_{i=1}^{\ell}\vone_{\{\tau_{i}(v)\le t\}}.
\end{align*}
One can easily check that
$$\E V_{n}^{(\ell)}(t) = n^{d}\pr_{\pi}(\tau_{1}(0)>t)^{\ell} \text{ and } \E R_{n}^{(\ell)}(t) = n^{d}\pr_{\pi}(\tau_{1}(0)\le t)^{\ell}.$$
In particular when $t/n^{d}\to u$ we have $n^{-d}\E V_{n}^{(\ell)}(t) \to e^{-\ell u/G(0)}$ and $n^{-d}\E R_{n}^{(\ell)}(t) \to (1-e^{-u/G(0)})^{\ell}$ as $n\to\infty$ where
\begin{align*}
	G(\xi)=G_{1/2}(\xi)
	 & =\text{ expected number of visits to $\xi$}\notag\\
	 & \qquad\text{ for a $1/2$\ndash lazy random walk on $\dZ^{d}$ starting from } 0.
\end{align*}
The above result follows from standard literature, for instance see~\cite{BH91}*{Equation~(2.26)},~\cite{AB92}*{Theorem~1} or~\cite{TW2011}*{Proposition~3.7}. Moreover, for an $\eps$\ndash lazy random walk, one can easily check that $G_{\eps}(\cdot)=(1-\eps)^{-1}G_{\ast}(\cdot)$, where $G_{\ast}$ is the classical Green's function for the simple random walk on $\dZ^{d}$. If needed, we will use $G(0;\dZ^{d})$ instead of $G(0)$ to emphasize the dependence on the dimension $d$.

For $\xi\in\dZ^d_n$, we define
\begin{align*}
	g_n(\xi) & := \sum_{t=0}^\infty (\pr_0(X_t = \xi)-\pr_{\pi}(X_t = \xi))\\
	\text{ and }
	g_n'(\xi) & := \sum_{t=1}^\infty t(\pr_0(X_t = \xi)-\pr_{\pi}(X_t = \xi)).
\end{align*}
Note that, heuristically $g_n(0)$ is the difference between the expected number of visits to 0 by two $\frac12$\ndash lazy random walks on $\dZ^d_n$ starting from the origin and the uniform distribution, respectively, up to a large multiple of the mixing time. Moreover, for fixed $\xi\in\dZ^{d}$, we have $g_{n}(\xi)\to G(\xi)$ as $n\to\infty$. One can easily check that
\[
	g'_n(0)=\sum_{\xi\in\dZ^d_n}g_n(\xi)^2 - g_n(0)>0.
\]
We remark that $g_n'(0)$ stays bounded in $d\ge 5$ and $g_n'(0)$ grows at rate $\log n$ when $d=4$, and at rate $n$ when $d=3$. See Lemma~\ref{lem:gnorder} for an upper bound on the growth rate of $g_n(\xi)$ and $g_n'(\xi)$. We write down the mean behavior with first order correction in the following Lemma.

\begin{lem}\label{lem:mean}
	Let $d\ge 3$ and $\ell\ge 1$ be fixed. Define $u:=(t+1)/n^{d} \in (0,\infty)$, then
	\begin{align*}
		\E(V_n^{(\ell)}(t))
		=n^d e^{-\frac{\ell u}{g_n(0)}} + \ell\left( \frac{ug_n'(0)}{g_n(0)^3}+\frac{u}{2g_n(0)^2}-\frac{g'_n(0)}{g_n(0)^2}\right) + O(n^{-d+3}).
	\end{align*}
\end{lem}

Our first main result is the following explicit variance for $V_{n}^{(\ell)}(t)$. Define
\begin{align}\label{def:var}
	\sigma_{n,\ell}^2(t):=\var(V_{n}^{(\ell)}(t)).
\end{align}
\begin{thm}\label{thm:var5}
	Let $d\ge 5$ and $\ell\ge 1$ be fixed. Assume that $t/n^{d}\to u\in (0,\infty)$. There exists a function $\nu_d:(0,\infty)\to (0,\infty)$ such that
	\begin{align*}
		\lim_{n\to\infty}\frac{1}{n^d}\gs_{n,\ell}^2(t) = \nu_d(2\ell u/G(0)).
	\end{align*}
	Moreover, we have an explicit formula for $\nu_{d}(u)$ given by
	\begin{align}\label{eq:nud}
		\nu_d(u)
		=e^{-u}\sum_{\xi \in \dZ^{d}} \left(\exp\left(\frac{u G(\xi)}{G(0)+G(\xi)}\right) -1- \frac{uG(\xi)}{G(0)+G(\xi)} + \frac{uG(\xi)^2(G(\xi)-\ind_{\{\xi=0\}})}{G(0)^2(G(0)+G(\xi))}\right).
	\end{align}
\end{thm}

Note that the limiting variance in~\eqref{eq:nud} is strictly positive as $e^{x}-1-x\ge x^{2}/2>0$ for $x>0$ and $G(0)>1$. Also, it is finite as $\sum_{\xi\in\dZ^{d}} G(\xi)^{2}<\infty$ for $d\ge 5$.

One can see from the proof of Theorem~\ref{thm:var5} that $n^{-d}\gs_{n,\ell}^2(t)$ is governed by $\sum_{\xi\in\dZ^d_n} g_n(\xi)^2$. This sum can be written as $g_n'(0)+g_n(0)$, whose growth rate is
\[
	h_d(n):=
	\begin{cases}
		n   & \text{ when } d=3,       \\
		\log n & \text{ when } d=4 \text{ and}, \\
		1   & \text{ when } d\ge 5.
	\end{cases}
\]Thus, it is expected that the correct scaling order for the variance is $n^4$ when $d=3$ and $n^4\log n$ when $d=4$, and heuristically the variance with right scaling converges to the sum of the second order term of the exponential in~\eqref{eq:nud}.
The following theorem affirms that this is indeed the case.

\begin{thm}\label{thm:var34}
	Let $d\in\{3,4\}$ and $\ell\ge 1$ be fixed. Assume that $t/n^{d}\to u\in (0,\infty)$. Then we have
	\begin{align*}
		\lim_{n\to\infty}\frac{1}{n^d h_d(n)}\gs_{n,\ell}^2(t) = \nu_d(2\ell u/G(0)).
	\end{align*}
	where $h_3(n) = n$, $h_4(n)=\log n$,
	\begin{align}\label{eq:nud34}
		\nu_d(u)=\frac12 \ga_d u^2 e^{-u},
	\end{align}
	and
	\begin{align*}
		\ga_3 :=\frac{9}{\pi^4\cdot G(0;\dZ^3)^2} \sum_{v\in\dZ^3} \norm{v}^{-4}, \quad
		\ga_4 := \frac{16}{\pi^4\cdot G(0;\dZ^4)^2}\lim_{n\to\infty}\frac{1}{\log n}\sum_{\|v\|\le n}\|v\|^{-4}.
	\end{align*}
\end{thm}

One can quickly check that the limit in $\ga_4$ exists and is finite. Note that this constant also appears in the fluctuation behavior for competing random walks in~\cite{Miller2013a}.

\begin{rem}
	One can extend the fluctuation behavior of the vacant set from discrete-time random walks to continuous-time random walks. Let $(T(s))_{s\ge 0}$ be a Poisson process with intensity 1, independent of $X_t$. Let $Y_s = X_{T(s)}$ and $\wt{V}_n^{(\ell)}(t)$ be the vacant set of $Y_s$ up to time $t$. Then, the variance of $\wt{V}_n^{(\ell)}(t)$ can be computed via conditioning with Lemma~\ref{lem:mean} and Theorems~\ref{thm:var5},~\ref{thm:var34}. Indeed, it follows that
	\begin{align*}
		\var(\wt{V}_n^{(\ell)}(t))
		= \E(\var(\wt{V}_n^{(\ell)}(t)\mid T(t)=s )) + \var(\E(\wt{V}_n^{(\ell)}(t)\mid T(t)=s) ),
	\end{align*}
	for $t\approx un^d$. Then, we apply the fluctuation results for $V_n^{(\ell)}(s)$.
\end{rem}

\begin{rem}
	We remark that the order of the variance $\gs_{n,\ell}^2(t)$ at $t\approx un^d$ is the same as that of the variance for competing random walks in~\cite{Miller2013a}. Unlike our proof, which relies on an elementary and analytic approach using the generating function of hitting probabilities, Miller~\cite{Miller2013a} made use of the conditioning argument at the mixing time to obtain the cancellation in the expansion of the variance, which leads to the precise asymptotic for the variance. It would be interesting to get a probabilistic proof and interpretation of our results on the asymptotic of the variance $\gs_{n,\ell}^2(t)$.
\end{rem}

We can generalize the above result to the size of the intersection of ranges and similar sets in the following way. For any subset $I\subseteq [\ell]$, we define
\begin{align*}
	R_{n,\ell}^I(t)
	:=\sum_{v\in\dZ^d_n} \prod_{i\in I}\vone_{\{\gt_i(v)\le t\}}
	\prod_{j\in I^c}\vone_{\{\gt_j(v)> t\}}
\end{align*}
as the number of vertices in $\dZ_{n}^{d}$ that are visited by walks indexed by $I$ but not by walks indexed by $[\ell]\setminus I$ at time $t$. We define the random vector indexed by $I\subseteq[\ell]$
$$
	\mvR_{n,\ell}(t):=(R_{n,\ell}^I(t))_{I\subseteq [\ell]}.
$$
Note that $R_{n,\ell}^\emptyset(t)=V_n^{(\ell)}(t)$ and $R_{n}^{[\ell]}(t)= R_n^{(\ell)}(t)$ as defined in~\eqref{def:v} and \eqref{def:r}. We can compute the variance-covariance structure for the random vector $\mvR_{n,\ell}(t)$ when $t\approx u n^{d}$.

\begin{thm}\label{thm:rcov}
	For $I,J\subseteq [\ell]$ and $t/n^d\to u\in(0,\infty)$, we have
	\begin{align*}
		&\lim_{n\to\infty}\frac1{n^d h_d(n)}\cov(R_{n,\ell}^I(t), R_{n,\ell}^J(t))\\
		&\qquad	=\sum_{m=0}^{\abs{I\cup J}}\theta_{k,r,m}(\exp(-u/G(0)) )\cdot \nu_d(2(\ell-m)u/G(0))
	\end{align*}
	where $k=\abs{I\cap J}, r=\abs{I\Delta J}$, $\nu_d(\cdot)$ is as given by~\eqref{eq:nud}\ndash\eqref{eq:nud34}, and
	\begin{align*}
		\theta_{k,r,m}(a)
		:=\sum_{j=(m-k)_+}^{r\wedge m} \binom{k}{m-j} \binom{r}{j}(1-2a)^{m-j} (-1)^{r-j}a^{j}.
	\end{align*}
\end{thm}
\begin{rem}
	Under suitable assumptions, our analysis could be applied to general vertex-transitive graphs with $g_n(0)$ bounded above. In particular, we consider the vacant set of an $\eps$\ndash lazy random walk on the hyper-cube $\dZ_2^n$. Since we have explicit formulas for the Green's function, the eigenvalues, and the eigenfunctions as in~\cite{ChungYau200a}*{Section 7}, we can apply our method to obtain the fluctuation behavior of the vacant set as $n\to\infty$. Suppose $\xi\in\dZ^n_2$ belongs to the $k$-th level for $0\le k\le n$, then the Green's function is given by
	\begin{align*}
		g_n(\xi)
		&= \frac{1}{(1-\eps)2^n}\sum_{v\in \dZ^n_2, v\neq 0}\frac{n}{\norm{v}_1}\cdot (-1)^{\langle v,\xi\rangle}\\
		& = \frac{1}{(1-\eps)2^n}\left(\sum_{i=1}^n \frac{n}{i}\binom{n}{i} - 2\sum_{i=1}^{k}\frac{1}{\binom{n-1}{i-1}}\sum_{j=i}^n\binom{n}{j}\right)
	\end{align*}
	see~\cite{Beveridge2016a}*{Section 4} and~\cite{ChungYau200a}*{Example 4}. In particular, $g_n(0)\to 2/(1-\eps)$ as $n\to\infty$. Similarly, using the spectral representation one can compute that $g'_n(0)\to 4/(1-\eps)^2-2/(1-\eps)$. Thus, $\sum_{\xi\in\dZ_2^n}g_n(\xi)^2 - g_n(0)^2\to 0$. Then, when $t/2^n\to u$, the variance of the size of the vacant set $V_n^{(\ell)}(t)$ is of order $2^n$ and
	\[
		2^{-n}\var(V_n^{(\ell)}(t))\to \nu((1-\eps)\ell u)
		\text{ with } \nu(x)=
		e^{-x}(e^{x/2} -1- (1-\eps)x/4).
	\]
	One might need to verify the detail somewhere else for other vertex-transitive graphs such as the Cayley graph of the symmetric group $S_n$.
\end{rem}

The first step in the proof of Theorem~\ref{thm:var5} and~\ref{thm:var34} is the following simplification, which follows from the transitivity of the graph and independence of the random walks:
\begin{align}\label{eq:var1}
	\var( V_{n}^{(\ell)}(t) ) & = n^d \sum_{\xi \in \dZ_{n}^{d}} \left(\pr(\tau(0,\xi)>t)^{\ell} -\pr(\tau(0)>t)^{2\ell}\right)
\end{align}
where
\begin{align*}
	\tau(0,\xi):= \inf\{t\ge 0: X_t\in\{0,\xi\}\}
\end{align*}
is the hitting time of the set $\{0,\xi\}$ for the random walk $X_t$. In particular, we need a precise estimate of the tail behavior for the hitting time.

\begin{lem}\label{lem:tailprob}
	For $d\ge 3$, $\xi\in\dZ^d_n$, and $u=(t+1)/n^d\in (0,\infty)$, we have
	\begin{align*}
		\pr_\pi(\gt(0,\xi)>t)
		&= e^{- \frac{u}{f_n(\xi)}}\left( 1+\frac{ u }{n^{d}}\left(\frac{f'_n(\xi)}{f_n(\xi)^3}+ \frac{1}{2f_n(\xi)^2} \right) -\frac{f_n'(\xi)}{n^d f_n(\xi)^2}\right)\\ 
		&\qquad \qquad+O(n^{-2d+3})
	\end{align*}
	where
	\begin{align}\label{def:fn}
		f_n(\xi) :=\frac12(g_n(0)+g_n(\xi)),
		\quad
		f_n'(\xi):=\frac12(g_n'(0)+g_n'(\xi)).
	\end{align}
\end{lem}

\begin{rem}
	It is easy to check that $0\le f'_n(\xi)\le f'_n(0)$ for all $\xi$ using the spectral representation~\eqref{eq:gnspec}. Moreover, $n^{-d}\sum_\xi f'_n(\xi)=\frac12f'_n(0).$
\end{rem}

\subsection{Related Literature}

We first review the relevant literature on the range of simple random walks on the square lattice $\dZ^d$. Dvoretzky and Erd\"os~\cite{DE51} proved the strong law of large numbers for the range on $\dZ^d$ for $d\ge 2$. A central limit theorem was obtained by Jain and Orey~\cite{JO68}. They showed that for strongly transient random walks, the variance of the range up to time $t$ is of order $t$, and the range with a suitable normalization converges to normal distribution. Note that a simple random walk on $\dZ^d$ is strongly transient if and only if $d\ge 5$. Later, Jain and Pruitt~\cite{JP70clt} considered the general transient case. It was shown in~\cite{JP70clt} that the variance of the range up to time $t$ is of order $t$ if $d=4$, and $t\log t$ if $d=3$, and that the central limit theorem holds for $d=3, 4$. The recurrent case ($d=2$) was studied by Le Gall~\cite{LeGall}, who proved that if $\cR_t$ is the range up to time $t$, then $(\cR_t-\E(\cR_t))/(t(\log t)^{-2})$ converges to the intersection local time of a planar Brownian motion, which is a non-Gaussian distribution.

There have been efforts to study the capacity of the range of random walks on $\dZ^d$. For a finite set $A\subset \dZ^d$ in $d\ge 3$, the capacity of $A$ is defined by the probability that a simple random walk starting from $A$ never returns to $A$. Jain and Orey~\cite{JO68} showed the law of large numbers for the capacity of the range of a simple random walk when $d\ge 5$. Later, Chang~\cite{Chang17} extended the result for $d=3,4$. Furthermore, it was shown in~\cite{Chang17} that when $d=3$, the capacity of the range converges to that of Brownian motion. Recently, Asselar, Schapira, and Sousi~\cites{ASS18zd, ASS19z4} derived the central limit theorems for the capacity of the range when $d\ge 6$ and $d=4$.

Sznitman~\cite{Sznitman2010} introduced the random interlacement model to study the trace left by a simple random walk in a discrete torus $\dZ^d_n$ for a time $n^d$. The random interlacement at level $u$, denoted by $\cI^u$, can be constructed via a Poisson point process on the set of doubly infinite nearest neighbor paths on $\dZ^d$ with intensity measure given in terms of the Newtonian capacity. The model provides the local picture of the trace left by a simple random walk in a discrete torus.

One of the applications of the random interlacement model is to investigate the percolative properties of the vacant set on $\dZ^d_n$ up to time proportional to the size of the discrete torus. It was shown in~\cite{Sznitman2010} that there exists a critical $u_\ast$ such that the vacant set does not percolate for large $u>u_\ast$ when $d\ge 3$, and percolates for small $u<u_\ast$ when $d\ge 7$. Later, Teixeira and Windisch~\cite{TW2011} showed that if $u>0$ is large enough, then the volumes of all the components of the vacant set are of order $(\log n)^{\gl(u)}$ and if $u>0$ is small enough then there exists a macroscopic component. They also proved that if $d\ge 5$, the macroscopic component is unique in the small $u$ regime. The proofs of these results are based on couplings of a simple random walk on $\dZ^d$ with the random interlacement.

Unlike the percolation of the vacant set, the fluctuation of the size of the vacant set does not have any threshold at the time level $u$. This is because the local behavior of the random walk is crucial in the study of the percolative property. At the same time, the variance of the size of the vacant set depends on both global and local pictures.

The random interlacement captures the local behavior of a simple random walk on the discrete torus, in a sense that one can be approximated by the other as $n\to\infty$ in a box of size smaller order than the size of the torus. Consider the vacant set left by the random interlacement $\cI^u$ in a box $[-n/2,n/2)^d$, say $W^u_n := [-n/2,n/2)^d \setminus \cI^u$. Since $\pr(\{x,y\}\subset W^u_n)=\exp(-2u/(G_\ast(0)+G_\ast(x-y)))$ where $G_\ast(x)$ is the Green's function for $\dZ^d$, the variance of the vacant set $W^u_n$ can be computed as
\begin{align*}
	\var(W^u_n) = n^d \sum_{\xi\in [-\frac{n}{2},\frac{n}{2})^d} \left( e^{-\frac{2u}{G_\ast(0)+G_\ast(\xi)}}-e^{-\frac{2u}{G_\ast(0)}} \right)+o(1).
\end{align*}
Note that the sum in the right-hand side does not converge as $n\to\infty$ for $d\ge 3$, while the sum in~\eqref{eq:nud} does.
Compared to Theorem~\ref{thm:var5}, one can see that the contribution from the global fluctuation matters in the variance computation of the vacant set $V^{(l)}_n(t)$.

The vacant set $\cV_n(t)$ at a time of order $n^d\log n^d$ was studied in~\cites{Al91, Bel, MS17}. The cover time $\tau_{\mathrm{cov}}$ is the maximum of the hitting times $\gt(\xi)$ over $\xi\in\dZ^d_n$. Let $t_{\mathrm{cov}}=\max_x\E_x[\gt_{\mathrm{cov}}]$. It is well-known that $t_{\mathrm{cov}}= C_d n^d\log n^d(1+o(1))$ as $n\to\infty$. Belius~\cite{Bel} proved that the fluctuations of $\gt_{\mathrm{cov}}$ are governed by the Gumbel distribution, in a sense that $\gt_{\mathrm{cov}}/(C_d n^d)-\log n^d$ converges to Gumbel in law as $n\to\infty$ for $d\ge 3$. He also showed that the scaling limit of the vacant set up to time $C_d n^d\log n^d$ as a set-valued process in $(\dR/\dZ)^d$ is indeed a Poisson point process. Heuristically, the Gumbel fluctuation of the cover time implies that, at a time of order $n^d\log n^d$, the hitting times $\gt(\xi)$, $\xi\in\dZ^d_n$, are approximately almost exponential and independent. A natural question is the limiting behavior of the vacant set $\cV_n(t)$ at a time $t=\alpha t_{\mathrm{cov}}(1+o(1))$ for $\alpha>0$. Miller and Sousi~\cite{MS17} showed that there are two thresholds $\alpha_0(d)<\alpha_1(d)$ such that $\cV_n(t)$ is approximately Bernoulli random variable indexed by $\dZ^d_n$ if $\alpha>\alpha_1(d)$, and the total variation distance between $\cV_n(t)$ and Bernoulli random variable is 1 if $\alpha<\alpha_0(d)$.

Miller~\cite{Miller2013a} investigated the fluctuation behavior of the trace by competing random walks. Consider $\ell$ independent simple random walks $(X_{t,i})$, $i\in[\ell]$, on $\dZ^d_n$. We assume that each site $\xi$ is painted by $i$ irreversibly at time $t$ if $X_{t,i}=\xi$. Let $\cA_i(t)$ be the set of sites painted by $i$ up to time $t$. Miller~\cite{Miller2013a} computed the limiting behavior of the variance of $\abs{\cA_i(\infty)}$ for $d\ge 3$ and $\ell=2$.
Indeed, it was shown that the order of the variance of $\abs{\cA_i(\infty)}$ has the same order as the sum of squares of the $\dZ^d$ Green's function, which coincides with our fluctuation behavior. He also extended the result to vertex-transitive graphs with some assumptions on the mixing time and provided precise limits of the variances for the hypercube and the Cayley graph of $S_n$ as a corollary.

\subsection{Roadmap}

The article is structured as follows. In Section~\ref{sec:pre} we provide notations, background details, and preliminary results for the later analysis. Section~\ref{sec:series} contains the proof for accurately computing coefficients from specific functions of a power series, which will play a crucial role in estimating the upper tail behavior for the hitting time of two-point sets. We provide proof of auxiliary results in Section~\ref{sec:proofA} and proof of the main theorems in Section~\ref{sec:proofM}. Finally, we conclude with a discussion and list of open problems in Section~\ref{sec:discuss}.

\section{Preliminaries}\label{sec:pre}

\subsection{Notations and Conventions}

For the rest of the article we use $X_t$ and $S_t$ for $\frac12$\ndash lazy random walks on $\dZ^d_n$ and $\dZ^d$, respectively. We explicitly write the dependence on $n,d$ when needed. For a set $A\subseteq \dZ_{n}^{d}$, the hitting time for $A$ is denoted by
$$
	\tau(A):=\inf\{t\ge 0: X_t\in A\}.
$$
We will add extra subscript $i$ when working with the $i$\ndash th lazy random walk. If $A=\{x\}$ or $A=\{x,y\}$, we simply write $\tau(x)$ and $\tau(x,y)$, instead of $\tau(A)$. Let $\pr$ and $\E$ be the probability and the expectation of $X_t$ starting from the uniform distribution $\pi_n$ on $\dZ_{n}^{d}$. We will use $\pr_{\xi}$ and $\E_{\xi}$ to denote the probability and the expectation of $X_t$ starting from $X_{0}=\xi$.
For $v=(v_1,v_2,\cdots, v_d)$ in $\dZ^d_n$ and $p \in[1,\infty)$, we use the notation
\[
	\norm{v}_p := \bigl( |v_1|^{p}\wedge(n-|v_1|)^p +|v_2|^{p}\wedge(n-|v_2|)^p +\cdots +|v_d|^{p}\wedge(n-|v_d|)^p\bigr)^{1/p}.
\]
For simplicity, we drop the subscript $p$ when $p=2$, that is, $\norm{v}=\norm{v}_2$.

We will write $a_n\lesssim b_n$, when there exists a finite positive constant $c$ such that $a_n\le c b_n$ for all $n$. Similarly, we will use $a_n\simeq b_n$, when $a_n \lesssim b_n$ and $b_n\lesssim a_n$. We will also use $a_n=O(b_n)$ or $b_n=\gO(a_n)$ for $a_n\lesssim b_n$ and $a_n=\Theta(b_n)$ for $a_n\simeq b_n$.

\subsection{Green's Function}

We recall the Green's functions and their basic properties on $\dZ^d$ and $\dZ^d_n$. For further detail, we refer~\cite{ChungYau200a}*{Section 7} for the discrete torus, and~\cites{Lawler, LawlerLimic} for the lattice.
The Green's function on $\dZ^d_n$ for a $\frac12$\ndash lazy simple random walk is defined by
\begin{align}\label{eq:gndef}
	g_n(\xi,\eta) = \sum_{t=0}^\infty (\pr_\xi(X_t=\eta)-\pi(\eta)).
\end{align}
For simplicity, we use the notation $g_n(\xi):= g_n(0,\xi)$ and drop the subscript $n$ when there is no ambiguity. For a $\frac12$\ndash lazy simple random walk $S_t$ on $\dZ^d$, the Green's function $G(\xi,\eta)$ is the expected number of visit to $\eta$ from $\xi$,
\begin{align*}
	G(\xi,\eta) = \E_{\xi}\left(\sum_{t=0}^\infty \vone_{\{X_t=\eta\}}\right) = \sum_{t=0}^\infty \pr_\xi(X_t=\eta).
\end{align*}
Let $G(\xi):=G(0,\xi)$.

The Green's function has the following spectral representation. The Laplacian matrix for the random walk is given by $\gD_{n}=I-P_{n}=\frac12(I-\frac{1}{2d}A_n)$ where $A_n$ is the adjacency matrix of $\dZ^d_n$, that is, $A_n(\xi,\eta)=1$ if $\xi$ is a neighbor of $\eta$ for $\xi,\eta\in\dZ^d_n$, and otherwise 0. Let $e_n(x):=\exp(2\pi i x/n)$, for $x\in \dR$. It is well-known that
\begin{align}
	\gf_v(\xi)=n^{-\frac{d}{2}}e_{n}(\la v,\xi\ra),\quad \xi\in \dZ^{d}_{n}\label{def:eigenfn}
\end{align}
for $v\in \dZ^{d}_{n}$ gives a complete set of orthonormal eigenfunctions for $\gD_n$ with the corresponding eigenvalues
\begin{align}
	\gl_v
	:=\frac12 \biggl( 1-\frac{1}{d}\sum_{j=1}^d \cos\bigl(2\pi v_j/n\bigr)  \biggr)
	=\frac{1}{d}\sum_{j=1}^d \sin^2(\pi v_j/n) \in [0,1].\label{def:eigen}
\end{align}
In particular, $\gD_{n}$ is diagonalizable and
\begin{align*}
	\gD_{n}(\xi,\eta) =\sum_{v\in\dZ_n^d} \gl_v \gf_v(\xi)\ol{\gf_v(\eta)}\quad \text{ for } \xi,\eta\in\dZ_{n}^{d}
\end{align*}
where $\bar{z}$ denotes the complex conjugate of $z\in\dC$.
Then, the Green's function can be written as
\begin{align}\label{eq:gnspec}
	g_n(\xi,\eta)
	= \sum_{v\in \dZ^{d}_{n}\setminus \{0\}} \frac{1}{\gl_v} \gf_v(\xi)\ol{\gf_v(\eta)}
	= n^{-d}\sum_{v\in \dZ^{d}_{n}\setminus \{0\}} \frac{1}{\gl_v}e_{n}(\la \xi-\eta, v\ra).
\end{align}
Using the spectral representation, one can see that the Green's function $G(\xi)$ for $\dZ^d$ ($d\ge 3$) is the limit of $g_n(\xi)$ as $n\to\infty$. Indeed, we have
\begin{align*}
	G(\xi)
	= \lim_{n\to\infty} g_n(\xi)
	= \int_{[0,1]^d}\frac{d}{\sum_{j=1}^d \sin^2(\pi x_j)}e^{2\pi i \xi\cdot x}\, dx \text{ for all } \xi\in\dZ^d.
\end{align*}
Let $\varphi_{d}(x):= \frac{d}{\sum_{j=1}^d \sin^2(\pi x_j)}$ for $x\in[0,1]^{d}$, then $G(\xi)$ is the Fourier transform of $\varphi$. By Plancherel's identity (\cite{Grafakos}*{Proposition 3.1.16}), if $d\ge 5$ then
\begin{align*}
	\int_{[0,1]^d}\varphi_{d}(x)^2\, dx
	=\sum_{\xi\in\dZ^d} G(\xi)^2.
\end{align*}
This fact will be used in the proof of the main results (see~\eqref{eq:sumgsquare}).

\subsection{Generating Function}

The Green's generating function for the torus is defined by
\begin{align}
	g_{n}(\xi,\eta; z)
	:= \sum_{t=0}^\infty \left( \pr_\xi(X_t=\eta) - \pi(\eta) \right)z^t
	\text{ for } z\in\dC, \abs{z}<1. \label{def:g_n}
\end{align}
For simplicity, we use $g_{n}(\xi;z)=g_{n}(0,\xi;z)$ and drop the subscript $n$ if there is no ambiguity. Note that the spectral representation provides
\begin{align}
	g_{n}(\xi,\eta; z)
	=\sum_{v\in \dZ^{d}_{n}\setminus \{0\}} \frac{1}{1-z\widehat{\gl}_v} \gf_v(\xi)\ol{\gf_v(\eta)}
	= n^{-d}\sum_{v\in \dZ^{d}_{n}\setminus \{0\}} \frac{e_{n}(\la \xi-\eta, v\ra)}{1-z\widehat{\gl}_v}
\end{align}
where $\widehat{\gl}_{v}:=1-\gl_{v}$. Also note that $g_{n}(\xi;z)$ is defined for $z\notin \{\widehat{\gl}_{v}^{-1} \mid v\in \dZ_n^d\setminus \{0\}\}$ and $g_n(\xi)=g_{n}(\xi;1)$. Moreover, the Green's generating function defined in~\eqref{def:g_n} satisfies
\begin{align*}
	\sum_{\xi\in\dZ_{n}^{d}} g_{n}(\xi;z)=0 \text{ for all } z.
\end{align*}

Let
\begin{align*}
	g_{n}'(\xi;1)
	:= \frac{d}{dz}g_{n}(\xi;z)\Big|_{z=1}
	&= \sum_{t=0}^\infty t\left( \pr_\xi(X_t=\eta) - \pi(\eta) \right)\\
	& = n^{-d} \sum_{v\in \dZ_n^d\setminus \{0\}} (1-\gl_v)\gl_v^{-2} e_{n}(\la \xi,v\ra).
\end{align*}
We simply denote by $g_n'(\xi) = g_n'(\xi;1)$. One can easily check that
\begin{align*}
	\sum_{\xi\in\dZ_n^d} g_{n}(\xi)^{2} = n^{-d} \sum_{v\in \dZ_n^d\setminus \{0\}} \gl_v^{-2} = g_n(0)+g_{n}'(0).
\end{align*}
For $\xi\in\dZ^d$ and $z\in\dC$ with $\abs{z}<1$, we define $G'(\xi;z)=\sum_{t=0}^\infty z^t\pr_0(S_t=\xi)$ and
\begin{align*}
	G'(\xi)=\frac{d}{dz}G(\xi;z)|_{z=1}=\sum_{t=0}^\infty t\pr_0(S_t=\xi)>0
\end{align*}
where $S_t$ is a $\frac12$\ndash lazy simple random walk on $\dZ^d$. Note that $\lim_{n\to\infty}g_n'(\xi)=G'(\xi)$ for each $\xi\in\dZ^d$ when $d\ge 5$. The Green's generating function has a probabilistic interpretation. Consider two independence $\frac12$\ndash lazy simple random walks on $\dZ^d$, $S_{t,1}$ and $S_{t,2}$ starting at $0$ and $\xi$ respectively. Then, $\E\abs{ \{(s,t)\mid S_{s,1}=S_{t,2}\} }$, expected size of the intersection of two random walk trajectories (counted with multiplicities), can be written as
\begin{align}\label{eq:Gsquaresum}
	\sum_{v\in\dZ^d} G(v)G(v-\xi) & = \sum_{t,s=0}^\infty\pr_{0,\xi}(S_{t,1}=S_{s,2})\notag              \\
	               & =\sum_{t,s=0}^\infty\sum_{v\in\dZ^d}\pr_{0}(S_{t,1}=v)\pr_{\xi}(S_{s,2}=v) \notag \\
	               & =\sum_{t,s=0}^\infty\pr_{0}(S_{t+s}=\xi)
	=\sum_{t=0}^\infty(t+1)\pr_{0}(S_{t}=\xi)
	=G(\xi)+G'(\xi).
\end{align}
In this paper, the main results are given in terms of the Green's function $G$ for a $\frac12$\ndash lazy random walk on $\dZ^d$. If needed, one can replace $G$ with the standard Green's function $G_\ast$ with minor modification. If $G_\eps(\xi)$ is the Green's function for the $\eps$\ndash lazy random walk on $\dZ^d$, $\eps\in(0,1)$, then one can see that
\begin{align*}
	G_\eps(\xi) = \frac{1}{1-\eps}G_\ast(\xi), \qquad
	G_\eps'(\xi) = \frac{\eps}{(1-\eps)^2}G_\ast(\xi)+\frac{1}{(1-\eps)^2}G_\ast'(\xi).
\end{align*}
In particular, we have $G(\xi)=2G_\ast(\xi)$ and $G'(\xi)=2G_\ast(\xi)+4G_\ast'(\xi)$.

The following Lemma~\ref{lem:gnorder} will be used to compute the growth rate of $g_{n}'(0)$ depending on $n$ and the dimension $d$ by taking $k=2$.
\begin{lem}\label{lem:gnorder}
	For $k\ge 1$, we have
	\begin{align*}
		\frac{1}{n^d}\sum_{v\in \dZ_n^d\setminus \{0\}} \gl_v^{-k}
		\simeq \begin{cases}
			1    & \text{ if } d>2k  \\
			\log n  & \text{ if } d=2k  \\
			n^{2k-d} & \text{ otherwise}.
		\end{cases}
	\end{align*}
	As a direct consequence, we have
	\begin{align*}
		\frac{d^k}{dz^k}g_n(0;z)\biggl|_{z=1}
		\simeq \begin{cases}
			1     & \text{ if } d>2k+2 \\
			\log n   & \text{ if } d=2k+2 \\
			n^{2k+2-d} & \text{ otherwise}.
		\end{cases}
	\end{align*}
\end{lem}
\begin{proof}
	From~\eqref{def:eigen}, it is easy to see that $\gl_v \simeq \norm{v}^2/(dn^2)$. Thus, we have
	\begin{align*}
		\sum_{v\in \dZ_n^d\setminus \{0\}} \gl_v^{-k}
		\simeq n^{2k} \sum_{v\in \dZ_n^d\setminus \{0\}} \norm{v}^{-2k}
		\simeq n^{2k} \sum_{r=1}^n r^{-2k}r^{d-1}
		\simeq \begin{cases}
			n^d    & \text{ if } d>2k  \\
			n^d\log n & \text{ if } d=2k  \\
			n^{2k}  & \text{ otherwise}.
		\end{cases}
	\end{align*}
	The second assertion follows from the fact that
	\begin{align*}
		\frac{d^k}{dz^k}g_n(0;z)\biggl|_{z=1}
		= \frac{k!}{n^{d}} \sum_{v\in\dZ_n^d\setminus\{0\}} \gl_v^{-(k+1)}(1-\gl_v)^{k}
	\end{align*}
	and this completes the proof.
\end{proof}

The next Lemma~\ref{lem:gndecay} tells us that $g_n(\xi)$ and $g'_n(\xi)$ converge to 0 as $\norm{\xi}$ becomes large, uniformly in $n$, which will be frequently used in the proofs of the main results. The proof is given in Section~\ref{sec:proofA}.
\begin{lem}\label{lem:gndecay}
	We have
	\begin{align*}
		g_n(\xi)
		 & =O(n^{-\min\{d-2, 2\}}+(1+\norm{\xi})^{-\min\{d-2, (d+2)/2\}})\quad \text{ for }d\ge 3, \\
		g'_n(\xi)
		 & =O(n^{-\min\{d-4, (d-2)/2\}}+(1+\norm{\xi})^{-\min\{d-4, d/2\}})\quad \text{ for }d\ge 5
	\end{align*}
	uniformly in $n, \xi\in\dZ_n^d$.
\end{lem}

Let
\begin{align}
	f_n(\xi;z) & :=\frac{1}{2}(g_n(0;z)+g_n(\xi;z))
	\qquad \text{ and }\qquad f_n(\xi) := f_n(\xi;1). \label{eq:fnxi}
\end{align}
The generating function for $\pr(\tau(0,\xi)>t)$ can be expressed in terms of the function $f_{n}(\xi;z)$ as follows (see~\cite{BH91}).

\begin{lem}\label{lem:series}
	For $\abs{z}<1$ and $\xi\in \dZ_{n}^{d}$, we have
	\begin{align}\label{eq:seriesexp}
		\sum_{t=0}^\infty z^t \pr(\tau(0,\xi)>t) =\frac{f_n(\xi;z)}{n^{-d}+(1-z)f_n(\xi;z)}.
	\end{align}
\end{lem}

From Lemma~\ref{lem:series}, finding a precise estimate on the hitting probability $\pr(\tau(0,\xi)>t)$, boils down to a refined analysis on the coefficients of the series expansion of the function on the right-hand side in~\eqref{eq:seriesexp}. We investigate the series expansion of such functions in Section~\ref{sec:series}. We give a proof of Lemma~\ref{lem:series} for completeness.

\begin{proof}[Proof of Lemma~\ref{lem:series}]
	Suppose $\xi\in\dZ^d_n\setminus\{0\}$ and $\abs{z}<1$. Let
	\[
		g^{+}_n(\xi;z):=\sum_{t=0}^\infty z^t \pr_0(X_t=\xi).
	\]
	Note that $g^{+}_n(\xi;z)=g_n(\xi;z)+n^{-d}(1-z)^{-1}$ and $\sum_{\xi\in\dZ^d_n}g^{+}_n(\xi;z)=(1-z)^{-1}$. For $x\in \dZ_{n}^{d}$, we have
	\begin{align*}
		g^{+}_n(x;z) + g^{+}_n(x+\xi;z)
		 & = \sum_{t=0}^\infty z^t \pr_0(X_t \in\{ x,x+\xi\})                         \\
		 & = \sum_{t=0}^\infty \sum_{s=0}^t z^t \pr_0(X_t \in\{ x, x+\xi \}, \tau(x,x+\xi) = s)        \\
		 & = \sum_{s=0}^\infty \sum_{t=0}^\infty z^{t+s} \pr_0(X_{t+s} \in\{ x, x+\xi \}, \tau(x,x+\xi) = s).
	\end{align*}
	Applying the Markov property at $\gt_x$ and $\gt_{x+\xi}$, we have
	\begin{align*}
		 & \pr_0(X_{t+s}=x, \tau(x,x+\xi) = s)                                \\
		 & \quad= \pr_0(X_{t+s}=x, \tau(x,x+\xi) = s=\gt(x)) +\pr_0(X_{t+s}=x, \tau(x,x+\xi) = s=\gt(x+\xi)) \\
		 & \quad= (\pr_0(X_t=0)+\pr_0(X_t=\xi))\pr_0(\gt(x, x+\xi)=s).
	\end{align*}
	Thus, we get
	\begin{align*}
		g^{+}_n(x;z) + g^{+}_n(x+\xi;z)
		 & = (g^{+}_n(0;z)+g^{+}_n(\xi;z)) \sum_{s=0}^\infty z^s \pr_0(\tau(x,x+\xi)=s).
	\end{align*}
	Averaging over $x$, we have
	\begin{align*}
		2n^{-d}(1-z)^{-1}            & = (2n^{-d}(1-z)^{-1} + 2f_n(\xi;z) ) \sum_{s=0}^\infty z^s \pr(\tau(0,\xi)=s) \\
		\text{ or }
		\sum_{s=0}^\infty z^s \pr(\tau(0,\xi)=s) & = \frac{n^{-d}}{n^{-d}+ (1-z)f_n(\xi;z)}.
	\end{align*}
	Therefore,
	\begin{align*}
		\sum_{t=0}^\infty z^{t} \pr(\tau(0,\xi)>t)
		 & = \sum_{t=0}^\infty \sum_{s=t+1}^{\infty}z^{t} \pr(\tau(0,\xi)=s)    \\
		 & = \sum_{s=1}^\infty \sum_{t=0}^{s-1}z^{t} \pr(\tau(0,\xi)=s)      \\
		 & = \sum_{s=0}^\infty (1-z^{s})(1-z)^{-1} \pr(\tau(0,\xi)=s)       \\
		 & =\frac{1}{1-z}\left( 1- \frac{n^{-d}}{ n^{-d}+(1-z)f_n(\xi;z) } \right)
		= \frac{f_n(\xi;z)}{ n^{-d}+(1-z)f_n(\xi;z)}.
	\end{align*}
	The same argument holds for the case $\xi=0$.
\end{proof}

From the uniform convergence of $g_n(\xi)$ and $\abs{g_n(\xi)}\le g_n(0)$, one can guess that $f_n(\xi)$ is uniformly away from 0 in $n$. The next lemma asserts that this is the case. The proof is given in Section~\ref{sec:proofA}.
\begin{lem}\label{lem:flb}
	There exists $C>0$ independent of $n$, such that $f_n(\xi)\ge C$ for all $\xi\in\dZ_n^d$.
\end{lem}

\section{Series expansion}\label{sec:series}
Fix an integer $k\ge 1$. Let $\ga_0,\ga_{i}, i\in[k]$ be a sequence of positive real numbers and $\gz_{0}:=1<\gz_{1}<\gz_{2}<\cdots <\gz_{k}$ be an increasing sequence of real numbers. Define the function
\begin{align*}
	f(z)=\sum_{i=1}^{k}\ga_{i}(\gz_{i}-z)^{-1} \text{ for } z\notin \{\gz_{i}, i\in[k]\}.
\end{align*}
We notice that $f$ is denoted by a generic function of such form only in this section and different from the one in~\eqref{eq:fnxi}.
\begin{lem}\label{lem:roots}
	The degree $k$ polynomial given by
	$$
		P_{k}(z):= \sum_{i=0}^{k}\ga_{i}\prod_{0\le j \le k, j\neq i} (\gz_{j}-z)
		=(\ga_{0}+(1-z)f(z)) \prod_{1\le j \le k} (\gz_{j}-z) $$
	has $k$ distinct real roots $\gc_{1}<\gc_{2}<\cdots<\gc_{k}$. Moreover, we have $1<\gc_{1}<\gz_{1}<\gc_{2}<\gz_{2}<\cdots <\gc_{k}<\gz_{k}$ and if $\gc:=1+\ga_{0}/f(1)<\gz_{1}$, then
	\begin{align*}
		1 + \frac{\ga_0}{f(\gc)}\le \gc_{1}\le 1+\frac{\ga_0}{f(1)}.
	\end{align*}
\end{lem}
\begin{proof}[Proof of Lemma~\ref{lem:roots}]
	We note that the function $\phi(z):=\ga_{0}(1-z)^{-1}+f(z)$ is continuous and strictly increasing in each of the interval $(\gz_{i-1},\gz_{i})$ for $i\in[k]$. Moreover, $\phi(\gz_{i}-)=\infty, \phi(\gz_{i}+)=-\infty$ for each $i=0,1,\ldots,k$. Thus there is a root of $\phi$ in the interval $(\gz_{i-1},\gz_{i})$, say $\gc_{i}$, for $i\in[k]$. Now, any root of $\phi$ is also a root of $P_{k}(z):=\phi(z)(1-z)\prod_{1\le j \le k} (\gz_{j}-z)$. Since $P_{k}$ is a degree $k$ polynomial, it has exactly $k$ roots and thus $\gc_{i},i\in[k]$ are all the roots of $P_{k}$.

	Note that, $f$ is strictly increasing in each of the interval $(1,\gz_1)$. Assume that, $\gc:=1+\ga_{0}/f(1)<\gz_{1}$. It is easy to check that $\hat\phi(z):=\ga_{0}+(1-z)f(z)$ satisfies $\hat\phi(1)=\ga_{0}>0$ and $\hat\phi(\gc)=\ga_{0}-\ga_{0}f(\gc)/f(1)<\ga_{0}-\ga_{0}=0$. Thus $\gc_{1}<\gc$.

	Moreover, $\ga_{0}+(1-\gc_{1})f(\gc_{1})=0$ implies that $\gc_{1}=1+\ga_{0}/f(\gc_{1})>1+\ga_{0}/f(\gc)$ as $f$ is increasing and we are done.
\end{proof}

Note that we have
\begin{align*}
	f(\gc_{i})=\frac{\ga_0}{\gc_i-1}>0, f'(\gc_{i})= \sum_{j=1}^{k}\ga_j(\gz_j-\gc_i)^{-2}>0\text{ for } i\in[k].
\end{align*}
In applications we have $\ga_{0}/f(1)\approx n^{-d}\ll \gz_{1}-1\approx n^{-2}$ for $d\ge 3$ and thus we have a good control on the first root $\gc_{1}\approx 1 + \ga_{0}/f(1)$.

Let $\hat{\ga}_{i}, i\in[k]$ be another sequence of positive real numbers. Define the function
\begin{align*}
	\hat{f}(z)=\sum_{i=1}^{k}\hat{\ga}_{i}(\gz_{i}-z)^{-1} \text{ for } z\notin \{\gz_{i}, i\in[k]\}
\end{align*}
and
\begin{align*}
	g(z) := \frac{\hat{f}(z)}{\ga_{0}+(1-z)f(z)} \text{ for }z\notin\{\gc_{i},i\in[k]\}.
\end{align*}
Here we define $g(\gz_{i}):=\frac{\hat{\ga}_{i}}{(1-\gz_i)\ga_{i}}$ so that $g$ is continuously differentiable everywhere except at $\gc_1,\gc_{2},\ldots,\gc_{k}$.

\begin{lem}\label{lem:decomp}
	We have
	\begin{align*}
		g(z) = \sum_{i=1}^{k} \frac{1}{\gc_{i}-z} \cdot \frac{\hat{f}(\gc_i)/f(\gc_i)}{1+\ga_0f'(\gc_i)f(\gc_i)^{-2}} \text{ for }z\notin\{\gc_{i},i\in[k]\}.
	\end{align*}
	In particular, the coefficient of $z^{t}, t\ge 0$ in the series expansion of $g$ around $0$ is given by
	\begin{align*}
		\sum_{i=1}^{k} \frac{\hat{f}(\gc_i)}{f(\gc_i)}\cdot \frac{\gc_{i}^{-t-1}}{1+ \ga_0f'(\gc_i)f(\gc_i)^{-2}}
		= \frac{\hat{f}(\gc_1)}{f(\gc_1)}\cdot \frac{\gc_1^{-t-1}}{1+ \ga_0f'(\gc_1)f(\gc_1)^{-2}} + e(t)
	\end{align*}
	where
	\begin{align*}
		\abs{e(t)} \le \max_{i\in[k]}\abs{\frac{\hat{f}(\gc_i)}{f(\gc_i)}} \cdot \gz_{1}^{-t}.
	\end{align*}
\end{lem}
\begin{proof}[Proof of Lemma~\ref{lem:decomp}]
	We can write $g$ as a ratio of a degree $(k-1)$ and a degree $k$-polynomial, as
	\begin{align*}
		g(z) = \frac{\hat{f}(z) \prod_{i=1}^{k}(\gz_i-z)}{(\ga_{0}+(1-z)f(z)) \prod_{i=1}^{k}(\gz_i-z)}
		=\frac{ \sum_{i=1}^{k}\hat\ga_{i}\prod_{1\le j \le k, j\neq i} (\gz_{j}-z)}{ \sum_{i=0}^{k}\ga_{i}\prod_{0\le j \le k, j\neq i} (\gz_{j}-z)}.
	\end{align*}
	The denominator has $k$ distinct real roots given by $\gc_{i},i\in[k]$. Thus we can write
	\begin{align*}
		g(z)=\sum_{i=1}^{k} a_{i}(\gc_{i}-z)^{-1}
	\end{align*}
	for some real numbers $a_{i},i\in[k]$. Moreover, we have
	\begin{align*}
		a_{i}=\lim_{z\to \gc_{i}} (\gc_{i}-z)g(z) = \frac{-\hat{f}(\gc_i)}{(\ga_{0}+(1-z)f(z))'|_{z=\gc_i}}
		= \frac{\hat{f}(\gc_i)}{f(\gc_i) +(\gc_i-1)f'(\gc_i)}.
	\end{align*}
	Finally we used the fact that $f(\gc_{i})=\frac{\ga_0}{\gc_i-1}>0$.

	The coefficient result follows since $(\gc_{i}-z)^{-1}=\sum_{t=0}^{\infty} \gc_{i}^{-t-1}z^{t}$ for $\abs{z}<1$ and the bound on $e(t):= \sum_{i=2}^{k} a_{i}\cdot \gc_{i}^{-t-1}$ follows as $\gc_{i}>\gz_{1}>1$ for all $i>1$ and
	\begin{align*}
		\abs{e(t)}\le \sum_{i=2}^{k} \abs{a_i}\cdot \gc_i^{-t-1}
		\le \gz_{1}^{-t}\cdot \max_{i\in[k]}\abs{\frac{\hat{f}(\gc_i)}{f(\gc_i)}} \cdot \sum_{i=1}^{k} \frac{\gc_{i}^{-1}}{1+ \ga_0f'(\gc_1)/f(\gc_1)^2}
	\end{align*}
	and the last sum is $f(0)/(\ga_{0}+f(0))< 1$.
\end{proof}

In our case, we have $\gc_{1}\approx 1 + \Theta(\ga_{0}), \ga_{0}=\Theta(n^{-d}), t=\Theta(n^{d}), \gz_{1}=1+\Theta(n^{-2})$, thus the error term is $\exp(-\Theta(n^{d-2}))$ whereas the first term is $\Theta(1)$.

\section{Proofs of Auxiliary Results}\label{sec:proofA}

\subsection{Proof of Lemma~\ref{lem:gndecay}}
By symmetry and translation invariance, it suffices to assume that $0\le \xi_j\le \frac{n}{2}$ for all $j=1,2,\ldots,d$. Recall that $\pi(\xi) = n^{-d}$ is the uniform distribution. Let $T_{\text{mix}}$ be the mixing time. It is well known (see~\cite{LPW}*{Chapter 4}) that $T_{\text{mix}}=O(n^2)$ and there exist $\bar{\gc}>0, c>0$ such that
\begin{align*}
	\sup_{\xi\in\dZ_n^d}\abs{\pr_0(X_t=\xi)-\pi(\xi)} = O(n^{-d} e^{-\bar{\gc}t/n^2})
\end{align*}
for $t\ge cn^2$. In particular,
\begin{align*}
	\sum_{t=cn^2}^\infty \abs{\pr_0(X_t=\xi)-\pi(\xi)} &= O(n^{-d+2})\\
	\text{ and }\sum_{t=cn^2}^\infty t\abs{\pr_0(X_t=\xi)-\pi(\xi)} & = O(n^{-d+4}).
\end{align*}
Clearly, $\pr_0(X_t=\xi)=0$ for $t<\norm{\xi}_1$. Thus, what is left is to control the contribution for $\norm{\xi}_1\le t\le cn^2$.

Let $S_t$ be a $\frac12$\ndash lazy random walk on $\dZ^d$. We can couple the two random walks $(S_t,X_t)$ on $\dZ^d, \dZ_n^d$, respectively, by defining $X_t=S_t\mod n$, coordinate-wise. Then
\begin{align*}
	\pr_0(X_t=\xi) = \sum_{k\in\dZ^d, \norm{\xi+nk}_1 \le t} \pr_0(S_t=\xi+nk).
\end{align*}
In time $t=O(n^2)$, the random walk $S_t$ can visit upto distance of order $n$ with high probability and $\norm{\xi+nk}_1=\gO(\max\{n,\norm{\xi}_1\})$ for $k\neq 0$.
In fact, from the local limit theorem (see~\cite{Lawler}*{Theorem 1.2.1}), it follows that there exists $K>0$ such that
\begin{align*}
	\pr_0(S_t=\xi+nk)
	\lesssim \left(\frac{d}{2\pi t} \right)^{\frac{d}{2}} e^{-\frac{d\norm{\xi+nk}^2}{2t}} + \norm{\xi+nk}^{-2}t^{-\frac{d}{2}}
\end{align*}
for $t\ge K$. Assume that $\norm{\xi}_1\ge K$. We have
\begin{align*}
	I_1:=\sum_{t=\norm{\xi}_1}^{cn^2}\sum_{ \substack{ k\in\dZ^d \\ \norm{\xi+nk}_1 \le t, \norm{k}\le\sqrt{d} }} \norm{\xi+nk}^{-2} t^{-\frac{d}{2}}
	\lesssim \norm{\xi}^{-2}\sum_{t=\norm{\xi}_1}^{cn^2} t^{-\frac{d}{2}}
	\simeq \norm{\xi}^{-\frac{d+2}{2}}.
\end{align*}
If $\norm{k}>\sqrt{d}$, using $\norm{\xi/n}\le \sqrt{d}/2$, we have $\frac{1}{2}\norm{k}<\norm{k+ \xi/n}<2\norm{k}$ and
\begin{align*}
	\sum_{ \substack{ k\in\dZ^d                    \\ \norm{\xi+nk}_1 \le t, \norm{k}>\sqrt{d} }} \norm{\xi+nk}^{-2}
	 & \lesssim n^{-2} \sum_{k\in\dZ^d, \norm{k} <Ct/n} \norm{k}^{-2}
	\simeq n^{-d}t^{d-2}.
\end{align*}
Thus
\begin{align*}
	I_2:=\sum_{t=\norm{\xi}_1}^{cn^2}\sum_{ \substack{ k\in\dZ^d  \\ \norm{\xi+nk}_1 \le t, \norm{k}>\sqrt{d} }} \norm{\xi+nk}^{-2} t^{-\frac{d}{2}}
	 & \simeq n^{-d}\sum_{t=\norm{\xi}_1}^{cn^2}t^{\frac{d}{2}-2}
	\simeq n^{-2}.
\end{align*}
On the other hand,
\begin{align*}
	I_3:=\sum_{t=\norm{\xi}_1}^{cn^2}\sum_{ \substack{ k\in\dZ^d \\
			\norm{\xi+nk}_1 <t, \norm{k}\le \sqrt{d} } } t^{-\frac{d}{2}} e^{-\frac{d\norm{\xi+nk}^2}{2t}}
	& \lesssim \sum_{t=\norm{\xi}_1}^{cn^2} t^{-\frac{d}{2}} e^{-\frac{d\norm{\xi}^2}{2t}}\\
	& \simeq \int_0^\infty t^{-\frac{d}{2}} e^{-\frac{d\norm{\xi}^2}{2t}}
	\simeq \norm{\xi}^{2-d}
\end{align*}
and
\begin{align*}
	I_4
	& :=\sum_{t=\norm{\xi}_1}^{cn^2}\sum_{ \substack{ k\in\dZ^d \\ \norm{\xi+nk}_1 <t, \norm{k}> \sqrt{d} } } t^{-\frac{d}{2}} e^{-\frac{d\norm{\xi+nk}^2}{2t}}\\
	& \lesssim \sum_{t=\norm{\xi}_1}^{cn^2}\sum_{k\in\dZ^d, \norm{k} <Ct/n } t^{-\frac{d}{2}} e^{-\frac{dn^2\norm{k}^2}{8t}}
	\lesssim n^{2-d}.
\end{align*}
In the last inequality, we used the fact that $\sum_{k\in\dZ^d} s^{-\frac{d}{2}}e^{-C\|k\|^2/s}\lesssim 1$. Therefore,
\begin{align*}
	\abs{g_n(\xi)}
	 & \le \sum_{t=cn^2 }^\infty \abs{\pr_0(X_t=\xi)-\pi(\xi)} + \sum_{t=0}^{cn^2} \pi(\xi) + \sum_{t=\norm{\xi}_1}^{cn^2} \pr_0(X_t=\xi) \\
	 & \lesssim n^{2-d} + I_1+ I_2+ I_3+ I_4                                               \\
	 & \lesssim n^{2-d} + \norm{\xi}^{-\frac{d+2}{2}} + \norm{\xi}^{2-d} + n^{2-d} + n^{-2}.
\end{align*}

We now prove the second assertion. Since $g_n(\xi;z) = \sum_{t=0}^\infty z^t(\pr_0(X_t=\xi)-\pi(\xi))$, $g_n'(\xi)$ can be written as
\begin{align*}
	g_n'(\xi) =\frac{d}{dz}g_n'(\xi;z)\biggr|_{z=1} =\sum_{t=0}^\infty t(\pr_0(X_t=\xi)-\pi(\xi)).
\end{align*}
As before, it suffices to estimate the summation over $\norm{\xi}_1\le t\le cn^2$. Let $\gb\in(1,2)$. Then
\begin{align*}
	J_1
	&:=\sum_{t=\norm{\xi}_1}^{cn^2}\sum_{\substack{k\in\dZ^d                                   \\ \norm{\xi+nk}_1 <n^\gb}} \norm{\xi+nk}^{-2} t^{-\frac{d}{2}+1}\\
	 & \lesssim \norm{\xi}^{-2}\sum_{t=\norm{\xi}_1}^{cn^2} t^{-\frac{d}{2}+1} + n^{-2} \sum_{t=0}^{cn^2}\sum_{\substack{k\in\dZ^d \\ \norm{k} <Cn^{\beta-1}}} \norm{k}^{-2} t^{-\frac{d}{2}+1}\\
	 & \lesssim \norm{\xi}^{-\frac{d}{2}}+ n^{(2-d)(2-\beta)}
\end{align*}
and
\begin{align*}
	J_2
	&:=\sum_{t=\norm{\xi}_1}^{cn^2}\sum_{\substack{k\in\dZ^d         \\ \norm{\xi+nk}_1 <n^{\gb}}} t^{-\frac{d}{2}+1} e^{-\frac{\norm{\xi+nk}^2}{t}}\\
	 & \lesssim
	\sum_{t=\norm{\xi}_1}^{cn^2} t^{-\frac{d}{2}+1} e^{-\frac{d\norm{\xi}^2}{2t}}
	+ \sum_{t=\norm{\xi}_1}^{cn^2}\sum_{\substack{k\in\dZ^d          \\
	\norm{k} <Cn^{\gb-1} }} t^{-\frac{d}{2}+1} e^{-\frac{dn^2\norm{k}^2}{8t}} \\
	 & \lesssim \norm{\xi}^{4-d} + n^{4-d}.
\end{align*}
By~\cite{Lawler}*{Lemma 1.5.1}, we get
\begin{align*}
	J_3:=\sum_{t=\norm{\xi}_1}^{cn^2} \sum_{ n^\gb \le \norm{\xi+nk}_1 < t } t\pr_0(S_t=\xi+nk)
	 & \lesssim \sum_{t=\norm{\xi}_1}^{cn^2} \sum_{n^\gb \le \norm{\xi+nk}_1<t } t\pr_0(\norm{S_t}\ge cn^{\gb +1}) \\
	 & \lesssim \sum_{t=\norm{\xi}_1}^{cn^2} \sum_{n^\gb \le \norm{\xi+nk}_1 <t} t e^{-\frac{cn^{\gb +1}}{t}} \\
	 & \lesssim n^{d+4} e^{-cn^{\gb-1}}.
\end{align*}
Thus, the local central limit theorem yields that for large $\norm{\xi}$,
\begin{align*}
	|g_n'(\xi)|
	 & \le \sum_{t=cn^2 }^\infty t\abs{\pr_0(X_t=\xi)-\pi(\xi)} + \sum_{t=0}^{cn^2} t\pi(\xi) + \sum_{t=\norm{\xi}_1}^{cn^2} t\pr_0(X_t=\xi) \\
	 & \lesssim n^{4-d} + J_1+ J_2+ J_3                                                   \\
	 & \lesssim n^{4-d} + n^{(2-d)(2-\gb)}+ n^{d+4} e^{-cn^{\gb-1}} + \norm{\xi}^{-\frac{d}{2}} + \norm{\xi}^{4-d}
\end{align*}
as desired.

\subsection{Proof of Lemma~\ref{lem:flb}}
Take $\eps=G(0)/4>0$. By Lemma~\ref{lem:gndecay}, there exist positive integers $N_1$ and $K$ such that
\[
	\abs{g_n(0) - G(0)}\le \eps,
	\abs{g_n(\xi)}\le \eps \text{ for all } \norm{\xi}\ge K, \xi\in\dZ_n^d, n\ge N_1.
\]
Moreover, we can choose $N\ge N_1$, such that
\[
	\sup_{\norm{\xi}\le K} \abs{g_n(\xi)-G(\xi)}\le \eps\text{ for all } n\ge N.
\]
Fix $n\ge N$. Recall that $f_n(\xi)=(g_n(0)+g_n(\xi))/2$. For $\xi\in\dZ_n^d$ with $\norm{\xi}\ge K$ we have
\begin{align*}
	f_n(\xi)
	\ge \frac{1}{2} G(0) - \frac{1}{2}\abs{g_n(0) - G(0)} - \frac{1}{2}\abs{g_n(\xi)}
	\ge \eps.
\end{align*}
For $\xi\in\dZ_n^d$ with $\norm{\xi}\le K$ we have
\begin{align*}
	f_n(\xi)\ge \frac{1}{2} (G(0)+G(\xi)) - \frac{1}{2}\abs{g_n(0) - G(0)} - \frac{1}{2}\abs{g_n(\xi)-G(\xi)}
	\ge \eps.
\end{align*}

\subsection{Proof of Lemma~\ref{lem:tailprob}}
Let $k:=\abs{\{\gl_{v}\mid v\in\dZ_{n}^{d}\}}-1$. We order the elements of the set $\{1/\widehat{\gl}_v\mid v\in\dZ_{n}^{d}\setminus\{0\}\}$ as $\gz_{1}<\gz_{2}<\cdots<\gz_{k}$. Clearly we have,
\begin{align*}
	\gz_{1}
	=\frac{1}{1-\min_{v\in\dZ_{n}^{d}\setminus\{0\}}\gl_v}
	= \frac{1}{1-d^{-1}\sin^{2}(\pi/n)}
	= 1+ \frac{\pi^2}{dn^2}(1+o(1))
\end{align*}
where the minimum is achieved at $\pm e_{i}, i\in [d]$.
One can also see that
\begin{align*}
	f_n(\xi;z) =\sum_{i=1}^{k}\ga_{i}(\xi)(\gz_{i}-z)^{-1}
\end{align*}
where $\ga_i(\xi)\ge 0$ and $\ga_1(\xi)\neq 0$. We fix $\xi\in\dZ^d_n$ and simply write $\ga_i(\xi)=\ga_i$. Let $\alpha_0=\frac{1}{n^d}$ and $\gz_0=1$. By Lemma~\ref{lem:roots}, there exist the distinct roots $\gc_1(\xi)<\gc_2(\xi)<\cdots<\gc_k(\xi)$ for the equation
\begin{align*}
	P_{k}(z)
	:= \sum_{i=0}^{k}\ga_{i}\prod_{0\le j \le k, j\neq i} (\gz_{j}-z)
	=(\ga_{0}+(1-z)f_n(\xi;z)) \prod_{1\le j \le k} (\gz_{j}-z)
	=0.
\end{align*}
It then follows from Lemma~\ref{lem:series} and Lemma~\ref{lem:decomp}, for $t=\Theta(n^{d})$ and $f=\wh{f}$, that
\begin{align}
	\pr(\tau(0,\xi)>t)
	 & = \sum_{i=1}^k\frac{\gc_i(\xi)^{-t-1}}{1+ n^{-d}f'_n(\xi;\gc_i(\xi))f_n(\xi;\gc_i(\xi))^{-2}} \nonumber \\
	 & = \frac{\gc_1(\xi)^{-t-1}}{1+ n^{-d}f'_n(\xi;\gc_1(\xi))f_n(\xi;\gc_1(\xi))^{-2}} + e(t)\label{eq:prest}
\end{align}
where $\abs{e(t)} \le \gz_{1}^{-t} \le e^{-\Theta(n^{d-2})}$ and $ \gc_1(\xi) = 1+\frac{n^{-d}}{f_n(\xi;\gc_1(\xi))}$. From now on, we use $r_n(\xi)=\gc_1(\xi)$ to emphasize its dependence on $n$.

Let $z=r_n(\xi)$, then $z-1 = \frac{1}{n^d f_n(\xi;z)}$. We use the notations $f_n(\xi)=f_n(\xi;1)$ and $f_n'(\xi)=f_n'(\xi;1)$. There exist $\wt{z}$ and $\ol{z}$ between 1 and $\gc_1$ such that
\begin{align*}
	f_n(\xi;z)
	 & = f_n(\xi) + f'_n(\xi)(z-1) + \frac{1}{2}f_n''(\xi;\wt{z})(z-1)^2                                                \\
	 & = f_n(\xi) + \frac{f'_n(\xi)}{n^d f_n(\xi)}\left( 1+\frac{f'_n(\xi;\ol{z})}{n^d f_n(\xi)f_n(\xi;z)} \right)^{-1} +\frac{f_n''(\xi;\wt{z})}{2n^{2d}f_n(\xi;z)^2}.
\end{align*}
Note that, $\abs{f''(\xi;z)}\le |f''_n(0)|\le n^{-d}\sum_{v\in\dZ_n^d\setminus\{0\}}\abs{\gl_v}^{-3}$ for $\abs{z}\le 1$.
By Lemma~\ref{lem:gnorder} with $k=3$, we have $f_n''(\xi;\wt{z})= O(n^{3})$. It then follows from Lemma~\ref{lem:flb} that
\begin{align}\label{eq:fnest}
	f_n(\xi;r_n(\xi))
	= f_n(\xi)\left( 1+ \frac{f'_n(\xi)}{n^d f_n(\xi)^2} +O(n^{3-2d}) \right).
\end{align}
Applying~\eqref{eq:fnest} to $r_n(\xi)$, we get
\begin{align*}
	\log(r_n(\xi))
	 & = \log\left( 1+\frac{1}{n^d f_n(\xi;r_n(\xi))} \right)                                   \\
	 & = \log\left( 1+\frac{1}{n^d f_n(\xi)} \left(1- \frac{f'_n(\xi)}{n^{d}f_n(\xi)^2} + O(n^{3-2d}) \right) \right)       \\
	 & = \frac{1}{n^d f_n(\xi)} -\frac{1}{n^{2d}}\left(\frac{f'_n(\xi)}{f_n(\xi)^3}+ \frac{1}{2f_n(\xi)^2} \right) + O(n^{3-3d}),
\end{align*}
which yields
\begin{align*}
	r_n(\xi)^{-(t+1)}
	= e^{- \frac{u}{f_n(\xi)}}\left( 1+\frac{ u }{n^{d}}\left(\frac{f'_n(\xi)}{f_n(\xi)^3}+ \frac{1}{2f_n(\xi)^2} \right) + O(n^{3-2d}) \right).
\end{align*}
By~\eqref{eq:fnest} and~\eqref{eq:prest}, we conclude
\begin{align*}
	\pr(\tau(0,\xi)>t)
	 & = \frac{r_n(\xi)^{-t-1}}{1+ n^{-d}f'_n(\xi;r_n(\xi))f_n(\xi;r_n(\xi))^{-2}} + e(t)                                               \\
	 & = e^{- \frac{u}{f_n(\xi)}}\left( 1+\frac{ u }{n^{d}}\left(\frac{f'_n(\xi)}{f_n(\xi)^3}+ \frac{1}{2f_n(\xi)^2} \right) -\frac{f_n'(\xi)}{n^d f_n(\xi)^2}\right)+ O(n^{3-2d}).
\end{align*}

\section{Proofs of Main Results}\label{sec:proofM}

\subsection{Proof of Theorem~\ref{thm:var5}}
For simplicity, we will omit the subscript $n$ in $f_n,g_n$. By Lemma~\ref{lem:tailprob} and~\eqref{eq:var1}, we have
\begin{align*}
	n^{-d}\var( V_{n}^{(\ell)}(t) )
	 & = \sum_{\xi \in \dZ_{n}^{d}} (\pr(\tau(0,\xi)>t)^{\ell} -\pr(\tau(0)>t)^{2\ell})                                                                                          \\
	 & = \sum_{\xi \in \dZ_{n}^{d}} (e^{- \frac{\ell u}{f(\xi)}} -e^{- \frac{2\ell u}{f(0)}})+\frac{\ell}{n^{d}}\sum_{\xi \in \dZ_{n}^{d}} e^{- \frac{\ell u}{f(\xi)}}\left(u\left(\frac{f'(\xi)}{f(\xi)^3}+ \frac{1}{2f(\xi)^2} \right) -\frac{f'(\xi)}{f(\xi)^2}\right) \\
	 & \quad -2\ell e^{- \frac{2\ell u}{f(0)}}\left(u\left(\frac{f'(0)}{f(0)^3}+ \frac{1}{2f(0)^2} \right) -\frac{f'(0)}{f(0)^2}\right) +o(1).
\end{align*}
It follows from $\sum_{\xi\in\dZ^d_n} g(\xi)=0$ that
\begin{align*}
	&\sum_{\xi \in \dZ_{n}^{d}} (e^{- \frac{\ell u}{f(\xi)}} -e^{- \frac{2\ell u}{f(0)}})\\
	 &\qquad =e^{- \frac{2\ell u}{f(0)}}
	\left(\sum_{\xi \in \dZ_{n}^{d}} \left(e^{ \frac{\ell u g(\xi)}{f(0)f(\xi)}} -1-\frac{\ell u g(\xi)}{f(0)f(\xi)}\right)
	- \frac{\ell u }{f(0)^2}\sum_{\xi \in \dZ_{n}^{d}}\frac{ g(\xi)^2}{f(\xi)}\right).
\end{align*}
Recall that $f_n(\xi)\ge C>0$ for all $n$ and $\xi$ by Lemma~\ref{lem:flb}.
Since $e^t -1-t\le \frac12 t^2\max\{e^t, 1\}$ and $g(\xi)\le g(0)=f(0)$, we have
\begin{align*}
	e^{ \frac{\ell u g(\xi)}{f(0)f(\xi)}} -1-\frac{\ell u g(\xi)}{f(0)f(\xi)}
	 & \le \frac12 \left(\frac{\ell u g(\xi)}{f(0)f(\xi)}\right)^2\max\{e^{\frac{\ell u g(\xi)}{f(0)f(\xi)}}, 1\} \\
	 & \le \frac{\ell^2 u^2 }{2C^2}\max\{e^{\frac{\ell u }{C}}, 1\}g(\xi)^2.
\end{align*}
Since $\lim_{n\to\infty}g_n(\xi)=G(\xi)$ for each $\xi$ and
\begin{align}\label{eq:sumgsquare}
	\lim_{n\to\infty}\sum_{\xi\in\dZ^d_n} g(\xi)^2
	=\lim_{n\to\infty}\frac{1}{n^d}\sum_{v\in\dZ^d_n} \frac{1}{\gl_v^2}
	=\int_{[0,1]^d}\varphi_d(x)^2\, dx
	=\sum_{\xi\in\dZ^d} G(\xi)^2
	<\infty
\end{align}
for $d\ge 5$ where $\varphi_d(x)=d(\sum_{j=1}^d \sin^2(\pi x_j))^{-1}$, the dominated convergence theorem implies that
\begin{align*}
	\lim_{n\to\infty}\sum_{\xi \in \dZ_{n}^{d}} (e^{- \frac{\ell u}{f(\xi)}} -e^{- \frac{2\ell u}{f(0)}})
	 & =\sum_{\xi \in \dZ^{d}} \left(e^{ -\frac{2\ell u }{(G(0)+G(\xi))}} -e^{- \frac{2\ell u}{G(0)}}-2\ell u e^{- \frac{2\ell u}{G(0)}}\frac{G(\xi)}{G(0)^2}\right).
\end{align*}
By Lemma~\ref{lem:gndecay}, for any $\eps>0$, there exist $N$ and $K$ such that $\abs{f_n(\xi)-\frac12 G(0)}<\eps$ and $\abs{f'_n(\xi)-\frac12 G'(0)}<\eps$ for all $n\ge N$ and $\norm{\xi}\ge K$. Thus, one can see that
\begin{align*}
	&\frac{1}{n^{d}}\sum_{\xi \in \dZ_{n}^{d}} e^{- \frac{\ell u}{f(\xi)}}\left(u\left(\frac{f'(\xi)}{f(\xi)^3}+ \frac{1}{2f(\xi)^2} \right) -\frac{f'(\xi)}{f(\xi)^2}\right)\\
	&\qquad\qquad	\to
	2e^{- \frac{2\ell u}{G(0)}}\left(\frac{2uG'(0)}{G(0)^3}+ \frac{u}{G(0)^2} -\frac{G'(0)}{G(0)^2}\right)
\end{align*}
as $n\to\infty$. Therefore, it follows from~\eqref{eq:Gsquaresum} that
\begin{align*}
	\lim_{n\to\infty} n^{-d}\var( V_{n}^{(\ell)}(t) )
	 & =\sum_{\xi \in \dZ^{d}} \left(e^{ -\frac{2\ell u }{(G(0)+G(\xi))}} -e^{- \frac{2\ell u}{G(0)}}-2\ell u e^{- \frac{2\ell u}{G(0)}}\frac{G(\xi)}{G(0)^2}\right) \\
	 & \quad+2\ell u e^{- \frac{2\ell u}{G(0)}}\left(\frac{G'(0)}{G(0)^3}+ \frac{1}{2G(0)^2} \right)                                 \\
	 & =\nu_d(2\ell u/G(0))
\end{align*}
where
\begin{align*}
	\nu_d(u)
	:=e^{-u}\sum_{\xi \in \dZ^{d}} \left(\exp\left(\frac{u G(\xi)}{G(0)+G(\xi)}\right) -1- \frac{uG(\xi)}{G(0)}\right)
	+ ue^{-u}\left(\frac{G'(0)}{G(0)^2}+\frac{1}{2G(0)}\right).
\end{align*}
Using $G(0)+G'(0) = \sum_{\xi\in\dZ^d}G(\xi)^2$, we conclude
\begin{align*}
	e^{u}\nu_d(u)
	 & =\sum_{\xi \in \dZ^{d}} \left(\exp\left(\frac{u G(\xi)}{G(0)+G(\xi)}\right) -1- \frac{uG(\xi)}{G(0)} + \frac{uG(\xi)^2}{G(0)^2}\right)
	-\frac{1}{2G(0)}                                                                                    \\
	 & =\sum_{\xi \in \dZ^{d}} \left(\exp\left(\frac{u G(\xi)}{G(0)+G(\xi)}\right) -1- \frac{uG(\xi)}{G(0)+G(\xi)} + \frac{uG(\xi)^3}{G(0)^2(G(0)+G(\xi))}\right)- \frac{1}{2G(0)}     \\
	 & =\sum_{\xi \in \dZ^{d}} \left(\exp\left(\frac{u G(\xi)}{G(0)+G(\xi)}\right) -1- \frac{uG(\xi)}{G(0)+G(\xi)} + \frac{uG(\xi)^2(G(\xi)-\ind_{\{\xi=0\}})}{G(0)^2(G(0)+G(\xi))}\right).
\end{align*}

\subsection{Proof of Theorem~\ref{thm:var34}}
Note that $g_n'(0)\to \infty$ as $n\to\infty$ for $d=3,4$. Again, we will omit the subscript $n$ in $f_n,g_n$, for simplicity. We have
\begin{align*}
	\frac{1 }{n^d g'(0)}\var( V_{n}^{(\ell)}(t) )
	 & = \frac{1}{g'(0)}\sum_{\xi \in \dZ_{n}^{d}} (\pr(\tau(0,\xi)>t)^{\ell} -\pr(\tau(0)>t)^{2\ell})               \\
	 & = \frac{1}{g'(0)} \sum_{\xi \in \dZ_{n}^{d}} (e^{- \frac{\ell u}{f(\xi)}} -e^{- \frac{2\ell u}{f(0)}})
	+\frac{\ell}{n^{d}}\sum_{\xi \in \dZ_{n}^{d}}
	\left(\frac{ue^{- \frac{\ell u}{f(\xi)}}}{f(\xi)^3} -\frac{e^{- \frac{\ell u}{f(\xi)}}}{f(\xi)^2}\right) \frac{f'(\xi)}{g'(0)} \\
	 & \quad -2\ell \left(\frac{u e^{- \frac{2\ell u}{g(0)}}}{g(0)^3} -\frac{e^{- \frac{2\ell u}{g(0)}}}{g(0)^2}\right) +o(1).
\end{align*}
Here the $o(1)$ is $O(n^{3-d}/g'(0))$.
Let $\eps>0$. By Lemma~\ref{lem:gndecay}, there exist $N$ and $K$ such that for all $n\ge N$ and $\norm{\xi}\ge K$, $\abs{g(\xi)}<\eps$.
Note that there exist $C_1, C_2$ independent of $n$ and $\xi$ such that $0<C_1\le f(\xi)\le C_2<\infty$ by Lemma~\ref{lem:flb} and $\abs{g(\xi)}\le g(0)$. Consider
\[
	\psi(x) := u x^{-3}e^{-\ell u / x}- x^{-2}e^{-\ell u / x}
\]
for $x\in[C_1,C_2]$. One can easily see that $\psi$ is bounded and Lipschitz on $[C_1, C_2]$. Indeed, there exists $C_3$ such that $\abs{\psi(x)-\psi(y)}\le C_3(\abs{x-y}\wedge 1)$ and for $x,y\in [C_1, C_2]$ and for $i=1,2$. Thus,
\begin{align*}
	\abs{\psi(f(\xi))-\psi(g(0)/2)}
	 & =\abs{
		\left(\frac{ue^{- \frac{\ell u}{f(\xi)}}}{f(\xi)^3} -\frac{e^{- \frac{\ell u}{f(\xi)}}}{f(\xi)^2}\right)
		-\left(\frac{8ue^{- \frac{2\ell u}{g(0)}}}{g(0)^3} -\frac{4e^{- \frac{2\ell u}{g(0)}}}{g(0)^2}\right) }\\
	& \le C_3\abs{f(\xi)-\frac{1}{2}g(0)}.
\end{align*}
Note that
\begin{align*}
	\frac{1}{n^{d}}\sum_{\xi \in \dZ_{n}^{d}}
	\psi(g(0)/2) \frac{f'(\xi)}{g'(0)}
	=\frac{1}{2}\psi(g(0)/2)
	=\frac{4ue^{- \frac{2\ell u}{G(0)}}}{G(0)^3} -\frac{2e^{- \frac{2\ell u}{G(0)}}}{G(0)^2}.
\end{align*}
because $\sum_{\xi\in\dZ^d_n}g'(\xi)=0$, and that $0\le f'(\xi)\le g'(0)$ for all $n,\xi$. Since $\abs{ f(\xi)-\frac{1}{2}g(0) }<\eps/2$ for all $n\ge N$, we get
\begin{align*}
	\left|\frac{1}{n^d}\sum_{\xi \in \dZ_{n}^{d}}
	\left(\psi(f(\xi))-\psi(g(0)/2)\right) \frac{f'(\xi)}{g'(0)}\right|
	 & \quad\le
	\frac{\eps C_3}{2n^{d}}\sum_{\xi \in \dZ_{n}^{d}, \norm{\xi}\ge K} \frac{f'(\xi)}{g'(0)}
	+
	\frac{C_3}{n^{d}}\sum_{\xi \in \dZ_{n}^{d}, \norm{\xi}< K} \frac{f'(\xi)}{g'(0)} \\
	 & \quad\le
	\frac{\eps C_3}{2} +\frac{C_4 K^d}{n^{d}},
\end{align*}
which yields in turn that
\begin{align*}
	\lim_{n\to\infty}\frac{\ell}{n^{d}}\sum_{\xi \in \dZ_{n}^{d}}
	\left(\frac{ue^{- \frac{\ell u}{f(\xi)}}}{f(\xi)^3} -\frac{e^{- \frac{\ell u}{f(\xi)}}}{f(\xi)^2}\right) \frac{f'(\xi)}{g'(0)}
	 & = 2\ell e^{- \frac{2\ell u}{G(0)}} \left(\frac{2u}{G(0)^3} -\frac{1}{G(0)^2}\right).
\end{align*}

On the other hand, if
\begin{align}\label{def:w}
	w_n = w = 2\ell u / g_n(0),
\end{align}
then
\begin{align*}
	\sum_{\xi \in \dZ_{n}^{d}} (e^{- \frac{\ell u}{f(\xi)}} -e^{- \frac{2\ell u}{f(0)}})
	 & = e^{- w}\sum_{\xi \in \dZ_{n}^{d}} (e^{- \frac{w g(\xi)}{g(0)+g(\xi)}} -1).
\end{align*}
It follows from Lemma~\ref{lem:flb} and $\abs{e^{s}-1-s-s^2/2}\le C(a)|s|^3$ for $|s|\le a$, that
\begin{align*}
	\left|e^{- \frac{w g(\xi)}{g(0)+g(\xi)}} -1-\frac{w g(\xi)}{g(0)+g(\xi)}-\frac{w^2 g(\xi)^2}{2(g(0)+g(\xi))^2}\right|
	\le \frac{w^3 \abs{g(\xi)}^3}{6(g(0)+g(\xi))^3}
	\le \frac{w^3}{C} \abs{g(\xi)}^3.
\end{align*}
We claim that
\begin{align}\label{claim1}
	\lim_{n\to\infty}\frac{1}{g'(0)}\sum_{\xi\in\dZ^d_n} \abs{g(\xi)}^3 =0.
\end{align}
Indeed, this follows from
\begin{align*}
	\frac{\sum_{\xi\in\dZ^d_n} \abs{ g(\xi)}^3}{\sum_{\xi\in\dZ^d_n} g(\xi)^2}
	 & \le
	\frac{\sum_{\xi\in\dZ^d_n, \norm{\xi}<K} \abs{g(\xi)}^3}{\sum_{\xi\in\dZ^d_n} g(\xi)^2}
	+\eps\frac{\sum_{\xi\in\dZ^d_n, \norm{\xi}\ge K} g(\xi)^2}{\sum_{\xi\in\dZ^d_n} g(\xi)^2}
	\le
	\frac{C K^d}{\sum_{\xi\in\dZ^d_n} g(\xi)^2}
	+\eps<2\eps
\end{align*}
for large $n$, and $g'(0) = \sum_{\xi\in\dZ^d_n} g(\xi)^2 - g(0) \to \infty$ as $n\to\infty$. Thus, we obtain
\begin{align*}
	\lim_{n\to\infty}\frac{1}{g'(0)}\sum_{\xi\in\dZ^d_n}\left(e^{- \frac{w g(\xi)}{g(0)+g(\xi)}} -1-\frac{w g(\xi)}{g(0)+g(\xi)}-\frac{w^2 g(\xi)^2}{2(g(0)+g(\xi))^2}\right) =0.
\end{align*}
Using $\sum_{\xi\in\dZ^d_n} g(\xi) = 0$, one can write
\begin{align*}
	\sum_{\xi\in\dZ^d_n}\frac{ g(\xi)}{g(0)+g(\xi)}
	 & = -\sum_{\xi\in\dZ^d_n}\frac{ g(\xi)^2}{g(0)(g(0)+g(\xi))}\\
	&= \sum_{\xi\in\dZ^d_n}\frac{ g(\xi)^3}{g(0)^2(g(0)+g(\xi))}-\sum_{\xi\in\dZ^d_n}\frac{ g(\xi)^2}{g(0)^2}.
\end{align*}
By Lemma~\ref{lem:flb} and the claim~\eqref{claim1}, we conclude that
\begin{align*}
	\lim_{n\to\infty}\frac{1}{g'(0)}\sum_{\xi\in\dZ^d_n}\frac{ g(\xi)}{g(0)+g(\xi)}
	 & = -\frac{ 1}{G(0)^2}.
\end{align*}
Here, we used the fact that $\frac{1}{g'(0)}\sum_{\xi\in\dZ^d_n} g(\xi)^2\to 1$ as $n\to\infty$. Similarly,
\begin{align*}
	\lim_{n\to\infty}\frac{1}{g'(0)}\sum_{\xi\in\dZ^d_n}\frac{ g(\xi)^2}{(g(0)+g(\xi))^2}
	 & = \frac{ 1}{G(0)^2}.
\end{align*}
Thus, for $w$ as defined in~\eqref{def:w} we have
\begin{align*}
	\lim_{n\to\infty}\frac{e^{-w}}{g'(0)}\sum_{\xi \in \dZ_{n}^{d}} (e^{- \frac{w g(\xi)}{g(0)+g(\xi)}} -1)
	 & = 2\ell e^{-\frac{2\ell u}{G(0)}}\left(\frac{\ell u^2}{G(0)^4} - \frac{u}{G(0)^3}\right).
\end{align*}
Therefore,
\begin{align*}
	\lim_{n\to\infty}\frac{\var( V_{n}^{(\ell)}(t) ) }{n^d g_n'(0)}
	 & = 2\ell e^{-\frac{2\ell u}{G(0)}}\left(\frac{\ell u^2}{G(0)^4} - \frac{u}{G(0)^3}\right)
	+2\ell e^{- \frac{2\ell u}{G(0)}} \left(\frac{2u}{G(0)^3} -\frac{1}{G(0)^2}\right)        \\
	 & \qquad\qquad -2\ell e^{- \frac{2\ell u}{G(0)}}\left(\frac{u }{G(0)^3} -\frac{1}{G(0)^2}\right) \\
	 & =2\ell^2 u^2 e^{-\frac{2\ell u}{G(0)}}G(0)^{-4}.
\end{align*}
Let $\ga_d := \lim_{n\to\infty}\frac{g'(0)}{h_d(n) G(0)^{2}}$, then it follows from $\sum_{\xi\in\dZ^d_n}g_n(\xi)^2 = g_n(0)+g_n'(0)$, $g_n'(0)\to\infty$, and $g_n(0)<\infty$ for $d=3,4$ that
\begin{align*}
	\ga_d = \lim_{n\to\infty}\frac{1}{h_d(n) G(0;\dZ^d)^{2}}\sum_{\xi\in\dZ^d_n} g_n(\xi)^2.
\end{align*}
It follows from Lemma~\ref{lem:gnorder} that $\alpha_d$ is finite for $d=3,4$.
To simplify, $\ga_d$ further, first we consider the case $d=3$.

Using the spectral representation of $g_n(\xi)$ and~\eqref{def:eigen}, we see
\begin{align*}
	\frac{1}{n}\sum_{\xi\in\dZ^3_n} g_n(\xi)^2
	=\frac{1}{n^4}\sum_{v\in\dZ^3_n} \gl_v^{-2}
	=\frac{9}{\pi^4} \sum_{v\in\dZ^3_n} \left(\frac{n^2}{\pi^2}\sum_{j=1}^3 \sin^2(\pi v_j/n)\right)^{-2}.
\end{align*}
We split the summation over $v$ in the right hand side into two parts $\norm{v}\ge K_0$ and $\norm{v}<K_0$, where $K_0>0$ will be determined later. Then, we have
\begin{align*}
	\sum_{v\in\dZ^3_n, \norm{v}\ge K_0} \left(\frac{n^2}{\pi^2}\sum_{j=1}^3 \sin^2(\pi v_j/n)\right)^{-2}
	\le C \sum_{v\in\dZ^3_n, \norm{v}\ge K_0} \norm{v}^{-4}
	\le \frac{C}{K_0}.
\end{align*}
On the other hand, we choose $n$ large enough that
\begin{align*}
	\sum_{v\in\dZ^3_n, \norm{v}<K_0} \left(\frac{n^2}{\pi^2}\sum_{j=1}^3 \sin^2(\pi v_j/n)\right)^{-2}
	= \sum_{v\in\dZ^3_n, \norm{v}<K_0} \norm{v}^{-4} + o(1).
\end{align*}
By taking large enough $K_0$, we conclude
\begin{align*}
	\ga_3
	=\lim_{n\to\infty}\frac{g'_n(0)}{n G(0)^2}
	= \lim_{n\to\infty} \frac{1}{nG(0)^2}\sum_{\xi\in\dZ^3_n} g_n(\xi)^2
	= \frac{9}{\pi^4 G(0)^2} \sum_{v\in\dZ^3} \norm{v}^{-4}.
\end{align*}
Similarly, for $d=4$, we have
\begin{align*}
	\ga_4
	= \lim_{n\to\infty}\frac{1}{n^4\log n \cdot G(0)^2}\sum_{\xi\in\dZ^4_n} \gl_v^{-2}
	= \frac{16}{\pi^4 G(0)^2}\lim_{n\to\infty}\frac{1}{\log n }\sum_{v\in\dZ^4_n} \left(\frac{n^2}{\pi^2}\sum_{j=1}^4 \sin^2(\pi v_j/n)\right)^{-2}.
\end{align*}
Let $\eps\in(0,1)$. Using $|\sin t|\le C t$ for all $t$ and $\sin t = t + O(t^3)$ as $t\to 0$, we obtain
\begin{align*}
	\frac{1}{\log n}\sum_{\|v\|\ge n^{\eps}} \left(\frac{n^2}{\pi^2}\sum_{j=1}^4 \sin^2(\pi v_j/n)\right)^{-2}
	&\le \frac{C}{\log n}\sum_{\|v\|\ge n^{\eps}} \|v\|^{-4}\\
	& \lesssim \frac{1}{\log n} \int_{n^\eps}^{n} r^{-1}\, dr
	\lesssim \frac{\abs{\log \eps}}{\log n}
\end{align*}
and
\begin{align*}
	\lim_{n\to\infty}\frac{1}{\log n}\sum_{\|v\|\le n^{\eps}} \left(\frac{n^2}{\pi^2}\sum_{j=1}^4 \sin^2(\pi v_j/n)\right)^{-2}
	 & =\lim_{n\to\infty}\frac{1}{\log n}\sum_{\|v\|\le n^{\eps}} \|v\|^{-4}.
\end{align*}
Since this holds for any $\eps\in (0,1)$, we get
\begin{align*}
	\ga_4 = \frac{16}{\pi^4 G(0)^2}\lim_{n\to\infty}\frac{1}{\log n}\sum_{\|v\|\le n}\|v\|^{-4}.
\end{align*}

\subsection{Proof of Theorem~\ref{thm:rcov}}
A direct computation yields
\begin{multline*}
	n^{-d}\E(R_{n,\ell}^I(t)R_{n,\ell}^J(t))
	=\sum_{\xi\in\dZ^d_n}\pr(\gt(0,\xi)\le t)^{|I\cap J|}\pr(\gt(0,\xi)> t)^{|I^c \cap J^c|}\pr(\gt(0)\le t<\gt(\xi))^{|I\Delta J|}.
\end{multline*}
Let
\[
	|I\cap J|=k, \
	|I^c\cap J^c|=k', \
	|I\Delta J|=r, \
	\pr(\gt(0)>t)=a \text{ and }
	\pr(\gt(0,\xi)>t)=b.
\]
Then $k+k'+r=\ell$ and
\begin{align*}
	n^{-d}\E(R_{n,\ell}^I(t)R_{n,\ell}^J(t))
	 & =\sum_{\xi\in\dZ^d_n} (1-2a+b)^{k} b^{k'} (a-b)^{r}                               \\
	 & =\sum_{\xi\in\dZ^d_n} \sum_{i=0}^k \sum_{j=1}^{r}\binom{k}{i}\binom{r}{j}(1-2a)^i (-1)^{r-j}a^{j}b^{k-i+k'+r-j} \\
	 & =\sum_{\xi\in\dZ^d_n} \sum_{m=0}^{k+r} b^{\ell-m} \left(
	\sum_{j=(m-k)_+}^{r\wedge m} \binom{k}{m-j}\binom{r}{j}(1-2a)^{m-j} (-1)^{r-j}a^{j}
	\right)                                                      \\
	 & =\sum_{m=0}^{k+r} \theta_{k,r,m}(a) \sum_{\xi\in\dZ^d_n} b^{\ell-m}.
\end{align*}
Similarly, we have
\begin{align*}
	\E(R_{n,\ell}^I(t))\E(R_{n,\ell}^J(t))
	=n^{2d}\sum_{m=0}^{k+r} a^{2(\ell-m)} \theta_{k,r,m}(a)
\end{align*}
and so
\begin{align*}
	\cov(R_{n,\ell}^I(t), R_{n,\ell}^J(t))
	 & =\sum_{m=0}^{k+r}\theta_{k,r,m}(a)\cdot n^d\sum_{\xi\in\dZ^d_n}(b^{\ell-m}-a^{2(\ell-m)})\\
	& =\sum_{m=0}^{k+r}\theta_{k,r,m}(a) \cdot \gs_{n,\ell-m}^2(t).
\end{align*}
By Theorem~\ref{thm:var5}, we conclude
\begin{align*}
	\lim_{n\to\infty} \frac{1}{n^dh_d(n)} \cov( R_{n,\ell}^I(t), R_{n,\ell}^J(t) )
	=\sum_{m=0}^{\abs{I}+\abs{J}}\theta_{k,r,m}(e^{-\frac{u}{G(0)}}) \nu_d((\ell-m)u).
\end{align*}

\section{Discussion and Further Questions}\label{sec:discuss}

\begin{figure}[th]
	\centering
	\includegraphics[width=5.5in,page=1]{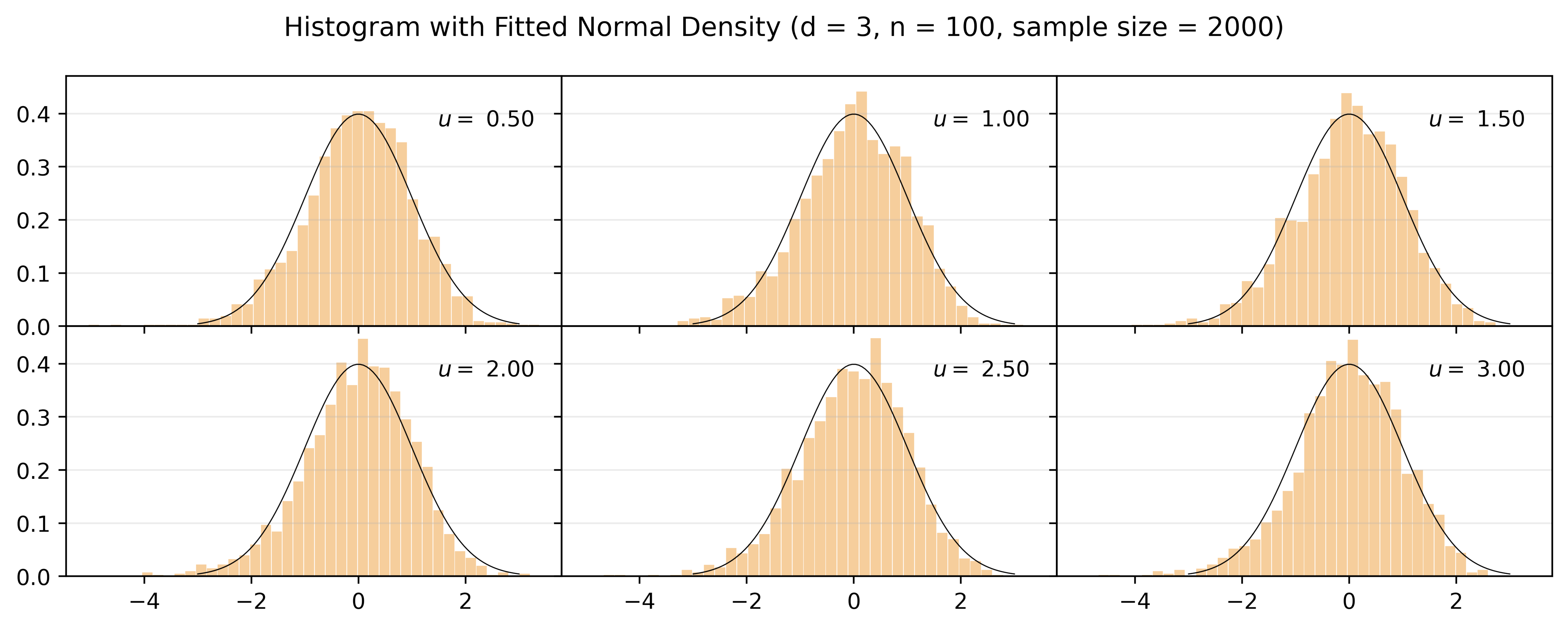}
	\includegraphics[width=5.5in,page=2]{data.pdf}
	\includegraphics[width=5.5in,page=3]{data.pdf}
	\caption{Histograms with fitted Gaussian density based on $2000$ samples in dimension $d=3,4$ and $5$, respectively.}
	\label{fig:data3}
\end{figure}

\begin{enumerate}
	\item As we have seen, our proofs are based on a precise tail probability of the hitting time at two points, which follows from an analytic observation on the generation function of hitting probabilities. A natural question is whether one can obtain an accurate tail behavior on the hitting time based on a probabilistic argument such as conditioning or constructing an appropriate coupling. Such an approach might give us an intuitive explanation of the asymptotic behavior and enable us to understand how the variance is created and grows.

	\item Once the limiting behavior of the mean and the variance of the size of the vacant set is known, it is natural to ask if the size of the vacant set with an appropriate normalization converges to a limiting distribution. Simulation results for $d=3,4,5$ (see Figure~\ref{fig:data3}) suggest that Gaussian central limit theorem hold for all $d\ge 3$ and $u>0$. It would also be interesting to figure out the covariance structure of the vacant sets at different times and see if the vacant set's size as a stochastic process indexed by $u$ converges to a Gaussian process.

	\item Instead of the size of the vacant set, one can consider the vacant set $\cV^{(\ell)}_n(un^d)$ as a set-valued process indexed by $u$. Then, it would be interesting to see the scaling limit of the vacant set up to a time proportional to the size of $\dZ^d_n$ as a subset of the continuum torus $[0,1]^d$. Note that up to the cover time of order $n^d\log n$, the scaling limit of the vacant set is studied by~\cites{Bel, MS17}.

	\item Miller~\cite{Miller2013a} studied the trace by two competing random walks on the discrete torus $\dZ^d_n$ until the torus is fully covered. It is natural to consider the trace at a time scale proportional to $n^d$. In particular, one can ask how the variance behaves as the time level grows and see if there is a time threshold for the limiting variance.
\end{enumerate}

\bibliography{rwrange.bib}
\end{document}